\documentclass[12pt,a4paper]{amsart}

\usepackage{latexsym} 
\usepackage[dvips]{graphics}
\usepackage{epsfig}
\usepackage{amssymb}
\usepackage{amsthm}
\usepackage{amsfonts}
\usepackage{amsmath}
\usepackage{amstext}
\usepackage{amscd}
\usepackage{epic}
\usepackage{eepic}

\numberwithin{equation}{section}

\theoremstyle{plain}
\newtheorem{thm}{Theorem}[section] 
\newtheorem{prop}[thm]{Proposition}

\newtheorem{lem}[thm]{Lemma}

\newtheorem{theorem*}{Theorem}[]

\theoremstyle{definition}

\newtheorem{example}[thm]{Example}

\theoremstyle{remark}
\newtheorem{rem}[thm]{Remark}

\newcommand{\C}{\mathbb{C}}

\newcommand{\R}{\mathbb{R}}
\newcommand{\K}{\mathbb{K}}

\begin{document}
\newcommand{\COMP}{\raisebox{0.1ex}{\scriptsize $\circ$}}
\title[Finiteness on Blow-semialgebraic triviality]
{Finiteness theorem on Blow-semialgebraic triviality\\
for a family of 3-dimensional algebraic sets}
\author{Satoshi KOIKE }
\address {Department of Mathematics, Hyogo University
of Teacher Education, 942-1 Shimokume, Kato,
Hyogo 673-1494, Japan}

\email {koike@hyogo-u.ac.jp}



\thanks{This research was partially supported by the Grant-in-Aid 
for Scientific Research (No. 18540084) of Ministry of Education, 
Science and Culture of Japan, JSPS Scientist Exchanges FY2006,
and JSPS Bilateral Joint Project FY2003 - FY2005. }

\newcommand{\abstracttext}{}

\maketitle


\begin{abstract}
In this paper we introduce the notion of ``Blow-semialgebraic
triviality consistent with a compatible filtration"
for an algebraic family of algebraic sets,
as an equisingularity for real algebraic singularities.
Given an algebraic family of 3-dimensional algebraic sets
defined over a nonsingular algebraic variety,
we show that there is a finite subdivision of the parameter
algebraic set into connected Nash manifolds over which
the family admits a Blow-semialgebraic trivialisation
consistent with a compatible filtration.
We show a similar result on finiteness also for a Nash family
of 3-dimensional Nash sets through the Artin-Mazur theorem.
As a corollary of the arguments in their proofs, 
we have a finiteness theorem on semialgebraic types of 
polynomial mappings from $\R^2$ to $\R^p$.

\end{abstract}

\vspace{4mm}


Consider a family of zero-sets of Nash mappings 
$f_t : N \to \R^k$ defined over a compact
Nash manifold $N$ (or a family of zero-sets of
Nash map-germs $f_t : (\R^n,0) \to (\R^k,0)$)
with a semialgebraic parameter space $J$.
Define  $F : N \times J \to \R^k$ by $F(x;t) = f_t(x)$.
Assume that $F$ is a Nash mapping.
For a subset $Q \subset J$, set $F_Q = F|_{N \times Q}$.
Then it is known that in the regular case a finiteness theorem
holds on Nash triviality for a family of Nash sets
$\{ (N,f_t^{-1}(0)) \}_{t \in J}$
(M. Coste - M. Shiota \cite{costeshiota}).
More precisely, if each $f_t^{-1}(0)$ does not 
contain a singular point of $f_t$,
then there is a finite partition of $J$ into 
Nash manifolds $Q_i$ such that
$\{ (N,f_t^{-1}(0)) \}_{t \in Q_i}$
is Nash trivial over each $Q_i$.
On the other hand, in the case of isolated singularities
a finiteness theorem holds on Blow-Nash triviality
for $\{ (N,f_t^{-1}(0)) \}_{t \in J}$
(T. Fukui - S. Koike - M. Shiota \cite{fukuikoikeshiota},
S. Koike \cite{koike1}).
Blow-Nash triviality is a notion introduced by the author
\cite{koike0, fukuikoikeshiota, koike1},
motivated by the work of Tzee-Char Kuo on blow-analyticity.
He establishes a locally finite classification theorem 
on blow-analytic equivalence for a family of analytic 
function germs with isolated singularities (\cite{kuo}).
In our Nash case, we have some result also for
non-isolated singularities (\cite{koike1}).
Namely, a finiteness theorem holds on Blow-semialgebraic
triviality for a family of 2-dimensional Nash sets
$\{ (N,f_t^{-1}(0)) \}_{t \in J}$
if the dimension of $N$ is 3 (or $n = 3$).
The results mentioned above for Nash sets are listed 
in the table of \cite{koike2}.

The finiteness results above are proved for a family of zero-sets
of Nash mappings defined over a compact Nash manifold
possibly with boundary.
Therefore the local case is covered with the ``compact" and
``with boundary" case.
In this paper we treat all the finiteness results
in the general case including also the non-compact case.
Namely, $N$ is a Nash manifold, not necessarily compact.
Let us remark that the local case is covered with the
non-compact case.

The first result in this paper is a Nash Isotopy Lemma
for finiteness property without the assumption of properness,
and we improve the finiteness theorems on the existence
of Nash trivial simultaneous resolution and on Blow-Nash triviality
to those in the general case using the new Isotopy Lemma.
We second give a programme to show finiteness
on Blow-semialgebraic triviality in the general case.
Then we show a finiteness theorem on Blow-semialgebraic 
triviality for a family of 2-dimensional Nash sets
in the case where the dimension $N$ is bigger than 3.
As a result, our list in \cite{koike2} is much more enriched.
We give the improved list (table (*)) in \S 4.
The main purpose in this paper is to show the theorem below, that is,
a finiteness theorem on Blow-semialgebraic triviality consistent with
a compatible filtration for a family of 3-dimensional algebraic sets.
For the definition of the new Blow-semialgebraic triviality,
see subsection 1.4.

Let $N$ be an affine nonsingular algebraic variety in $\R^m$,
and let $J$ be an algebraic set in $\R^a$.
Let $f_t : N \to \R^k$ ($t \in J$) be a polynomial mapping
such that $\dim f_t^{-1}(0) \le 3$ for $t \in J$.
Assume that $F$ is a polynomial mapping, 
i.e. $F$ is the restriction of a polynomial mapping
$\tilde{F} : \R^m \times \R^a \to \R^k$ to $N \times J$.
Then we have

\vspace{3mm}

\noindent {\bf Main Theorem.} {\em There exists a finite partition
$J = Q_1 \cup \cdots \cup Q_u$
which satisfies the following conditions}:

(1) {\em Each $Q_i$ is a Nash manifold which is Nash diffeomorphic
to an open simplex in some Euclidean space,
and $\dim f_t^{-1}(0)$ and $\dim f_t^{-1}(0) \cap S(f_t)$
are constant over $Q_i$.}

(2) {\em For each $i$ where $\dim f_t^{-1}(0) = 3$ and
$\dim f_t^{-1}(0) \cap S(f_t) \ge 1$ over $Q_i$, 
$F_{Q_i}^{-1}(0)$ admits a Blow-semialgebraic trivialisation 
consistent with a compatible filtration along $Q_i$.

In the case where $\dim f_t^{-1}(0) \le 2$ over $Q_i$ or
$\dim f_t^{-1}(0) \cap S(f_t) \le 0$ over $Q_i$,
$(N \times Q_i,F_{Q_i}^{-1}(0))$ admits
a trivialisation listed in table (*).}

\vspace{3mm}

Throughout this paper, $S(f)$ denotes the singular points set
of $f$ for a Nash mapping $f : N \to P$.

When we show a Blow-semialgebraic triviality for a family 
of algebraic sets or Nash sets,
the desingularisation theorem for an algebraic variety or a Nash variety
(H. Hironaka \cite{hironaka1, hironaka3}, E. Bierstone and
P.D. Milman \cite{bierstonemilman1, bierstonemilman2,
bierstonemilman3}) and the semialgebraic versions of Thom's Isotopy Lemmas 
(M. Shiota \cite{shiota2}) take very important roles in their proof.
But the 2nd Isotopy Lemma is not applicable directly to
a family of resolution maps, since a blow up is a typical example
of a non-Thom map.
The crucial point is how we establish a semialgebraic trivialisation
for a family of the restrictions of resolution maps to the exceptional sets.
In the case of Nash surfaces, the image of the restriction map is 
generically at most one-dimensional semialgebraic set.
Therefore we could show a finiteness theorem in \cite{koike1}, 
using the methods developed by T. Fukuda in \cite{fukuda1,fukuda2}. 
In the case of 3-dimensional algebraic sets, after being desingularised
the intersection of the strict transform of the algebraic set
and the exceptional set is generically a union of normal crossing
nonsingular algebraic surfaces.
In this paper, to overcome the crucial point, we apply
C. Sabbah's arguments discussed in the work \cite{sabbah2} on
``sans \'eclatement'' stratified analytic morphisms
based on the flattening theorem of Hironaka \cite{hironaka2,
hironaka4} and in the work \cite{sabbah1} on a local finiteness theorem 
for complex analytic mappings defined over a complex surface.

We first prepare several notions and some fundamental results
on Blow-semialgebraic triviality in \S 1, and describe
the aforementioned Nash Isotopy Lemma for finiteness property 
(Theorem I) in \S 2.
Then we describe the programme in \S 3 to show finiteness on 
Blow-semialgebraic triviality, which is applicable also 
in the non-compact case.
A part of the ideas of the programme is used in \cite{koike1}
to show some finiteness theorems in the compact case.
Here we divide our programme into eight processes.
In order to establish finiteness theorems on Blow-semialgebraic 
triviality with our method, it suffices to show only Process IV
related to the above crucial point, as the other processes always work.
Therefore we believe that it is natural to describe
our method as a programme in this paper.
As corollaries of the programme,
we have two finiteness theorems (Theorems III, IV) in \S 4.
In \S 5 we show a finiteness theorem for a family of the
main parts of 3-dimensional algebraic sets following our programme.
Here the main part of an algebraic set $V$ means the set of points
$x \in V$ at which the local dimension of $V$ at $x$ equals 
to the dimension of $V$.
Using the result shown in \S 5 and results listed in table (*),
we give a proof of our main theorem in \S 6.
In \S 7 we give a finiteness theorem on Blow-semialgebraic triviality
consistent with a Nash compatible filtration for a family of 3-dimensional 
Nash sets, corresponding to the algebraic result above.
In \S 8 we describe a finiteness theorem on semialgebraic types
of polynomial mappings from $\R^2$ to $\R^p$.

The first draft of this paper was written up
while the author was visiting the University of Sydney.
He would like to thank the institution for its support and hospitality.
He would also like to thank Karim Bekka, Takuo Fukuda, Shuzo Izumi,
Tzee-Char Kuo and Masahiro Shiota for useful communications,
and to Yeh Yu Chen for her encouragement.

\bigskip
\section{Preliminaries.}
\label{PRE}
\medskip

\subsection{Semialgebraic properties.}
A {\em semialgebraic set} of $\R^n$ is a finite union of the form
$$
\{ x \in \R^n \ | \ f_1(x) = \cdots = f_k(x) = 0, \
g_1(x) > 0, \cdots g_s(x) > 0 \}
$$
where $f_1, \cdots , f_k, g_1, \cdots , g_s$ are
polynomial functions on $\R^n$.
A $C^{\omega}$ submanifold of $\R^m$
is called a {\em Nash manifold}, 
if it is semialgebraic in $\R^m$.
In this paper, a submanifold always means a regular submanifold.
Let $M \subset \R^m$ and $N \subset \R^n$ be Nash manifolds.
A $C^{\omega}$ mapping $f : M \to N$ is called a {\em Nash mapping},
if the graph of $f$ is semialgebraic in $\R^m \times \R^n$.
We call the zero-set of a Nash mapping a {\em Nash set}.

Let $Q \subset \R^m$.
We say that $Q$ is a {\em Nash open simplex},
if it is a Nash manifold which is Nash diffeomorphic 
to an open simplex in some Euclidean space.

For $r = 1, 2, \cdots , \infty$, we can define the notions of a 
$C^r$ {\em Nash manifold}, a $C^r$ {\em Nash mapping}
and a $C^r$ {Nash open simplex} similarly.
By B. Malgrange \cite{malgrange}, a $C^{\infty}$ Nash manifold
is a Nash manifold and a $C^{\infty}$ Nash mapping 
between Nash manifolds is a Nash mapping.

\begin{thm}\label{tarski-seidenberg}
(Tarski-Seidenberg Theorem \cite{seidenberg}).
Let $A$ be a semialgebraic set in $\R^k$,
and let $f : \R^k \to \R^m$ 
be a semialgebraic mapping,
namely, the graph of $f$ is semialgebraic in 
$\R^k \times \R^m$.
Then $f(A)$ is semialgebraic in $\R^m$.
\end{thm}

\begin{thm}\label{lojasiewicz}
(Lojasiewicz's Semialgebraic Triangulation Theorem 
\cite{lojasiewicz1, lojasiewicz2}).
Given a finite system of bounded semialgebraic sets $X_{\alpha}$
in $\R^n$, there exist a simplicial decomposition
$\R^n = \cup_a C_a$ and a semialgebraic automorphism
$\tau$ of $\R^n$ such that

(1) each $X_{\alpha}$ is a finite union of some of the $\tau (C_a)$, 
 
(2) $\tau (C_a)$ is a Nash manifold in $\R^n$
and $\tau$ induces a Nash diffeomorphism $C_a \to \tau (C_a)$,
for every $a$.
\end{thm}

Concerning the triangulation theorem above,
we make an important remark.
We use the fact also in the proof of the Nash Isotopy Lemma
in the next section.

\begin{rem}\label{remark11}
In Theorem \ref{lojasiewicz} the boundedness is not essential.
Since there is a Nash embedding of $\R^n$ into $\R^{n+1}$
via $\R^n \subset S^n$, every semialgebraic subset in $\R^n$
can be considered as a bounded semialgebraic subset in $\R^{n+1}$.
\end{rem}

\begin{thm}\label{hardt}
(Hardt's Semialgebraic Triviality Theorem \cite{hardt}).
Let $B$ be a semialgebraic set, and let
$\Pi : \R^m \times B \to B$ be the projection.
For any semialgebraic subset $X$ of $\R^m \times B$,
there is a finite partition of $B$ into semialgebraic sets $N_i$,
and for any $i$, there are a semialgebraic set $F_i \subset \R^m$
and a semialgebraic homeomorphism
$$
h_i : F_i \times N_i \to X \cap \Pi^{-1}(N_i)
$$
compatible with the projection onto $N_i$.
\end{thm}

\subsection{Stratification.}
We recall the notions of the Whitney stratification
and the Thom mapping briefly in the $C^2$ Nash category. 
See \cite{gibsonetal, mather, thom, trotman1, trotman2, 
whitney1, whitney2} for the definitions of
the Whitney ($b$)-regularity and the Thom ($a_f$)-regularity.

Let $A \subset \R^m$ be a semialgebraic set. 
We say that a $C^2$-Nash stratification ${\mathcal S}(A)$ of $A$
is a {\em Whitney stratification},
if for any strata $X,\ Y$ with $\overline{X} \supset Y,\ X$
is Whitney ($b$)-regular over $Y$.

Let $A \subset \R^m$ and $B \subset \R^r$
be semialgebraic sets, and let $f : A \to B$ be a $C^2$-Nash mapping.
Assume that $A$ and $B$ admit $C^2$-Nash stratifications
${\mathcal S}(A)$ and ${\mathcal S}(B)$, respectively.
We call $f$ a {\em stratified mapping},
if for any stratum $X \in {\mathcal S}(A)$,
there is a stratum $U \in {\mathcal S}(B)$ such that
$f|_X : X \to U$ is a (onto) submersion.

Let $f : (A,{\mathcal S}(A)) \to (B,{\mathcal S}(B))$ be a $C^2$-Nash
stratified mapping.
We call $f$ a {\em Thom mapping},
if for any strata $X,\ Y$ of ${\mathcal S}(A)$
with $\overline{X} \supset Y,\ X$ is {\em Thom} ($a_f$)-{\em regular}
over $Y$.

\subsection{Nash trivial simultaneous resolution.}

Let $M, U$ be Nash manifolds, and let $V$ be a Nash set of $U$.
Let $\Pi : M \to U$ be a proper Nash modification.
We say that $\Pi$ is a {\em Nash resolution} of $V$ in $U$,
if there is a finite sequence of blowings-up
$\sigma_{j+1} : M_{j+1} \to M_j$ with smooth centres $C_j$
(where $\sigma_j, M_j$ and $C_j$ are of Nash class) such that:

(1) $\Pi$ is the composite of $\sigma_j$'s.

(2) The critical set of $\Pi$ is a union of Nash divisors
$D_1, \cdots , D_d$.

(3) $V^{\prime} $ (: the strict transform of $V$ in $M$ by $\Pi$)
is a Nash submanifold of $M$.

(4) $V^{\prime}, D_1, \cdots , D_d$ simultaneously have only normal
crossings.

(5) There is a thin Nash set $T$ in $V$ so that
$\Pi|_{\Pi^{-1}(V-T)} : \Pi^{-1}(V-T) \to V - T$
is a Nash isomorphism.

Concerning a Nash resolution, we have 
the following existence theorem.

\begin{thm}\label{desingularisation}
(H. Hironaka \cite{hironaka1, hironaka3}, E. Bierstone and 
P.D. Milman \cite{bierstonemilman1, bierstonemilman2, bierstonemilman3}).
For a Nash variety $V$ of a Nash manifold $U$,
there exists a Nash resolution of $V$ in $U$,
$\Pi : M \to U$.
\end{thm}

Let ${\mathcal M}, {\mathcal U}, I$ be Nash manifolds, 
and let ${\mathcal V}$ be a Nash set of ${\mathcal U}$.
Let $\Pi : {\mathcal M} \to {\mathcal U}$ be a proper Nash modification,
and let $q : {\mathcal U} \to I$ be an onto Nash submersion.
For $t \in I$, we set $U_t = q^{-1}(t), V_t = {\mathcal V} \cap U_t$
and $M_t = (q \circ \Pi)^{-1}(t)$.
We say that $\Pi$ gives a {\em Nash simultaneous resolution} 
of ${\mathcal V}$ in ${\mathcal U}$ over $I$,
if there is a finite sequence of blowings-up
$\tilde{\sigma}_{j+1} : {\mathcal M}_{j+1} \to {\mathcal M}_j$
with smooth centres $\tilde{C}_j$ (where $\tilde{\sigma}_j, 
{\mathcal M}_j$ and $\tilde{C}_j$ are of Nash class) such that:

(1) $\Pi$ is the composite of $\tilde{\sigma}_j$'s.

(2) The critical set of $\Pi$ is a union of Nash divisors
${\mathcal D}_1, \cdots , {\mathcal D}_d$.

(3) ${\mathcal V}^{\prime}$ (: the strict transform of ${\mathcal V}$
in ${\mathcal M}$ by $\Pi$) is a Nash submanifold of ${\mathcal M}$.

(4) ${\mathcal V}^{\prime}, {\mathcal D}_1, \cdots , {\mathcal D}_d$
simultaneously have normal crossings. The restrictions

\quad $q \circ \Pi|_{{\mathcal V}^{\prime}} : {\mathcal V}^{\prime} \to I$,

\quad $q \circ \Pi|_{{\mathcal D}_{j_1} \cap \cdots \cap {\mathcal D}_{j_s}} :
{\mathcal D}_{j_1} \cap \cdots \cap {\mathcal D}_{j_s} \to I$,

\quad $q \circ \Pi|_{{\mathcal V}^{\prime} \cap {\mathcal D}_{j_1} \cap \cdots
\cap {\mathcal D}_{j_s}} : {\mathcal V}^{\prime} \cap {\mathcal D}_{j_1} 
\cap \cdots
\cap {\mathcal D}_{j_s} \to I\ (1 \le j_1 < \cdots < j_s \le d)$

\leftline{are onto submersions.}

(5)  There is a thin Nash set ${\mathcal T}$
in ${\mathcal V}$ so that ${\mathcal T} \cap V_t$ is a thin set
in $V_t$ for each $t \in I$, and that
$\Pi|_{\Pi^{-1}({\mathcal V}-{\mathcal T})} 
: \Pi^{-1}({\mathcal V}-{\mathcal T})
\to {\mathcal V} - {\mathcal T}$ is a Nash isomorphism.

Let $\Pi : {\mathcal M} \to {\mathcal U}$ be a Nash simultaneous
resolution of a Nash variety ${\mathcal V} = \{ F = 0 \}$
in ${\mathcal U}$ over $I$, and let $t_0 \in I$.
We say that $\Pi$ gives a {\em Nash trivial simultaneous resolution}
of ${\mathcal V}$ in ${\mathcal U}$ over $I$,
if there is a Nash diffeomorphism 
$\phi : {\mathcal M} \to M_{t_0} \times I$ such that

(1) $(q \circ \Pi) \circ \phi^{-1} : M_{t_0} \times I \to I$
is the natural projection,

(2) $\phi ({\mathcal V}^{\prime}) = V_{t_0}^{\prime} \times I, \
\phi ({\mathcal D}_{j_1} \cap \cdots \cap {\mathcal D}_{j_s}) =
(D_{j_1,t_0} \cap \cdots \cap D_{j_s,t_0}) \times I$ and

\noindent $\phi ({\mathcal V}^{\prime} \cap {\mathcal D}_{j_1} \cap \cdots 
\cap {\mathcal D}_{j_s})
= (V_{t_0}^{\prime} \cap D_{j_1,t_0} \cap \cdots \cap D_{j_s,t_0})
\times I\ (1 \le j_1 < \cdots < j_s \le d)$.

\subsection{Blow-Nash triviality and Blow-semialgebraic triviality.}

Let $N$ and $Q$ be Nash manifolds, and let
$F : N \times Q \to \R^k$ be a Nash mapping.

Let $\Pi : \mathcal{M} \to N \times Q$ be a Nash trivial simultaneous
resolution of $F^{-1}(0)$ in $N \times Q$ over $Q$.
We say that $(N \times Q,F^{-1}(0))$ admits a
$\Pi$-{\em Blow-Nash trivialisation along} $Q$,
if the Nash trivialisation upstairs induces a semialgebraic
one of $(N \times Q,F^{-1}(0))$ over $Q$.

Let $\Pi : \mathcal{M} \to N \times Q$ be a Nash simultaneous
resolution of $F^{-1}(0)$ in $N \times Q$ over $Q$.
We say that $(N \times Q,F^{-1}(0))$ admits a
$\Pi$-{\em Blow-semialgebraic trivialisation along} $Q$,
if there is a $t$-level preserving semialgebraic homeomorphism
upstairs which induces a semialgebraic one of
$(N \times Q,F^{-1}(0))$ over $Q$.
We denote by $V^{\prime}$ be the strict transform of 
$V = F^{-1}(0)$ by $\Pi$, 
and by $MV$ the main part of $V$, that is,
the set of points $x \in V$ such that the local dimension 
of $V$ at $x$ equals to the dimension of $V$.
We further say that $MV$ admits 
a $\Pi$-{\em Blow-semialgebraic trivialisation along} $Q$,
if there is a semialgebraic trivialisation of 
$V^{\prime}$
upstairs which induces a semialgebraic one of
$MV$ over $Q$.

\subsection{Blow-semialgebraic triviality consistent with 
a compatible filtration.}
Let $N$ be a nonsingular algebraic variety,
let $Q$ be a Nash manifold,
and let $F : N \times Q \to \R^k$ be a polynomial mapping.
Set $V = F^{-1}(0)$.
Let $V = V^{(0)} \supset V^{(1)} \supset \cdots \supset V^{(r)}$
be a filtration of $V$ by algebraic subsets.
We call it a {\em compatible filtration} of $V$, if 

(1) $\dim V^{(0)} > \dim V^{(1)} > \cdots > \dim V^{(r)}$.

(2) For $0 \le j \le r - 1$, $V^{(j)} \setminus \bigcup_{i=0}^j MV^{(i)}$ 
is not empty and $V^{(j+1)} \supset V^{(j)} \setminus 
\bigcup_{i=0}^j MV^{(i)}$; $V = \bigcup_{i=0}^r MV^{(i)}$.

\noindent Incidentally we call the above filtration 
the {\em canonical filtration} of $V$,
if each $V^{j+1}$ is the Zariski closure of
$V^{(j)} \setminus \bigcup_{i=0}^j MV^{(i)}$, $0 \le j \le r - 1$.

Note that
$$
V = MV^{(0)} \cup (MV^{(1)} \setminus MV^{(0)}) \cup \cdots \cup
(MV^{(r)} \setminus \bigcup_{j=0}^{r-1} MV^{(j)}).
$$

We say that $V$ admits a {\em Blow-semialgebraic trivialisation 
consistent with a compatible filtration along} $Q$,
if there is a compatible filtration of $V$,
$V = V^{(0)} \supset V^{(1)} \supset \cdots \supset V^{(r)}$,
with the following properties:

There are a $t$-level preserving semialgebraic homeomorphism $\sigma$
trivialising $V$ over $Q$ and an algebraic simultaneous resolutions 
$\Pi^{(j)} : {\mathcal M}^{(j)} \to N \times Q$ of $V^{(j)}$ 
in $N \times Q$ over $Q$, $0 \le j \le r$,
such that $MV^{(j)}$ admits a $\Pi^{(j)}$-Blow semialgebraic 
trivialisation along $Q$, $\sigma |_{MV^{(0)}} = \sigma_0$ and
$\sigma |_{MV^{(j)} \setminus \cup_{i=0}^{j-1} MV^{(i)}} 
= \sigma_j |_{MV^{(j)} \setminus \cup_{i=0}^{j-1} MV^{(i)}}$ 
for $1 \le j \le r$.
Here $\sigma_j$, $0 \le j \le r$, is the semialgebraic
trivialisation of $MV^{(j)}$ induced by the $\Pi^{(j)}$-Blow 
semialgebraic trivialisation of $MV^{(j)}$ over $Q$.

Let $N$ and $Q$ be Nash manifolds, and let
$F : N \times Q \to \R^k$ be a Nash mapping.
Set $V = F^{-1}(0)$.
Then we can similarly define the notion of a compatible filtration
of $V$ in the Nash category.
We call it a {\em Nash compatible filtration} of $V$,
and we can define also the notion of
{\em Blow-semialgebraic triviality consistent with a Nash 
compatible filtration} for a family of Nash sets.
In the same way as the Zariski closure, for $S \subset \R^m$
we can define the notion of the {\em Nash closure} of $S$
as the smallest Nash set in $\R^m$ containing $S$.
Therefore we can define the notion of the {\em Nash canonical
filtration} of $V$ similarly.

A compatible filtration of an algebraic set is, by definition, 
a Nash compatible filtration of it.
But the canonical filtration of an algebraic
set is not always the Nash canonical filtration of it
(e.g. Example \ref{example1}).
We give an example to understand the notions of our filtrations
more clearly.

\begin{example}\label{example1}
Let $f : \R^2 \times \R \to \R$ be a polynomial function
defined by
$$
f(x,y,t) = (x^2 + (y^2 + t^2 - y^3)^2)(x^2 + (y + 1)^2 - 1).
$$
Set $V = f^{-1}(0)$.
Then $V$ consists of two irreducible algebraic subsets.
One is a cylinder and another is the union of a connected curve
and the origin contained in the cylinder.

\vspace{4mm}

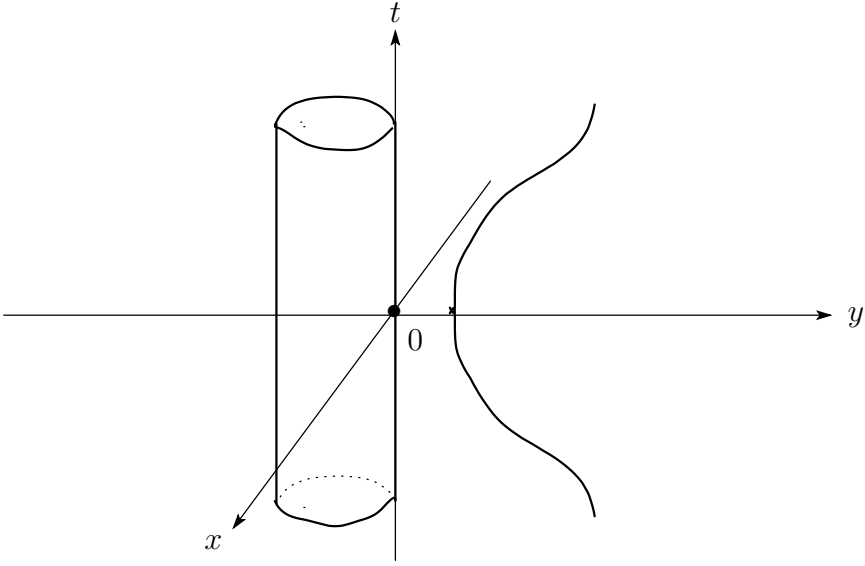
\unitlength 0.1in
\begin{picture}( 43.7000, 29.8000)(  7.9000,-41.1000)
%
\special{pn 8}%
\special{pa 2818 4110}%
\special{pa 2818 1348}%
\special{fp}%
\special{sh 1}%
\special{pa 2818 1348}%
\special{pa 2798 1416}%
\special{pa 2818 1402}%
\special{pa 2838 1416}%
\special{pa 2818 1348}%
\special{fp}%
\put(27.9000,-13.0000){\makebox(0,0)[lb]{$t$}}%
\put(51.6000,-28.9000){\makebox(0,0)[lb]{$y$}}%
\put(18.3000,-40.5000){\makebox(0,0)[lb]{$x$}}%
\put(28.8000,-30.1000){\makebox(0,0)[lb]{$0$}}%
%
\special{pn 8}%
\special{pa 2328 1822}%
\special{pa 2328 1822}%
\special{fp}%
\special{pa 2344 1846}%
\special{pa 2344 1846}%
\special{fp}%
\special{pa 2344 3834}%
\special{pa 2344 3834}%
\special{fp}%
%
\special{pn 13}%
\special{pa 2202 1830}%
\special{pa 2202 3810}%
\special{fp}%
%
\special{pn 13}%
\special{pa 2818 1822}%
\special{pa 2818 3800}%
\special{fp}%
%
\special{pn 13}%
\special{pa 2194 1846}%
\special{pa 2224 1856}%
\special{pa 2252 1872}%
\special{pa 2278 1890}%
\special{pa 2304 1908}%
\special{pa 2332 1924}%
\special{pa 2360 1938}%
\special{pa 2390 1948}%
\special{pa 2422 1956}%
\special{pa 2454 1962}%
\special{pa 2486 1964}%
\special{pa 2520 1966}%
\special{pa 2552 1968}%
\special{pa 2586 1968}%
\special{pa 2618 1966}%
\special{pa 2650 1960}%
\special{pa 2678 1950}%
\special{pa 2706 1936}%
\special{pa 2732 1918}%
\special{pa 2758 1898}%
\special{pa 2782 1874}%
\special{pa 2802 1854}%
\special{sp}%
%
\special{pn 13}%
\special{pa 2194 1838}%
\special{pa 2208 1808}%
\special{pa 2224 1782}%
\special{pa 2246 1758}%
\special{pa 2270 1740}%
\special{pa 2298 1724}%
\special{pa 2328 1714}%
\special{pa 2360 1704}%
\special{pa 2394 1698}%
\special{pa 2428 1694}%
\special{pa 2462 1692}%
\special{pa 2498 1690}%
\special{pa 2532 1690}%
\special{pa 2564 1692}%
\special{pa 2596 1694}%
\special{pa 2628 1698}%
\special{pa 2658 1704}%
\special{pa 2688 1714}%
\special{pa 2718 1728}%
\special{pa 2746 1744}%
\special{pa 2770 1764}%
\special{pa 2794 1786}%
\special{pa 2810 1814}%
\special{pa 2810 1830}%
\special{sp}%
%
\special{pn 13}%
\special{pa 2194 3800}%
\special{pa 2208 3828}%
\special{pa 2230 3852}%
\special{pa 2258 3872}%
\special{pa 2288 3884}%
\special{pa 2318 3892}%
\special{pa 2350 3898}%
\special{pa 2382 3906}%
\special{pa 2412 3914}%
\special{pa 2444 3922}%
\special{pa 2474 3930}%
\special{pa 2506 3932}%
\special{pa 2538 3930}%
\special{pa 2570 3922}%
\special{pa 2602 3912}%
\special{pa 2632 3902}%
\special{pa 2660 3894}%
\special{pa 2686 3884}%
\special{pa 2714 3870}%
\special{pa 2744 3846}%
\special{pa 2776 3810}%
\special{pa 2806 3782}%
\special{pa 2818 3786}%
\special{pa 2818 3792}%
\special{sp}%
%
\special{pn 8}%
\special{pa 2202 3818}%
\special{pa 2216 3788}%
\special{pa 2232 3762}%
\special{pa 2254 3738}%
\special{pa 2278 3720}%
\special{pa 2306 3706}%
\special{pa 2336 3694}%
\special{pa 2368 3684}%
\special{pa 2402 3678}%
\special{pa 2436 3674}%
\special{pa 2470 3672}%
\special{pa 2506 3670}%
\special{pa 2540 3670}%
\special{pa 2572 3672}%
\special{pa 2604 3674}%
\special{pa 2636 3678}%
\special{pa 2666 3684}%
\special{pa 2696 3694}%
\special{pa 2726 3706}%
\special{pa 2752 3724}%
\special{pa 2778 3744}%
\special{pa 2802 3766}%
\special{pa 2818 3794}%
\special{pa 2818 3810}%
\special{sp -0.045}%
%
\special{pn 8}%
\special{pa 790 2830}%
\special{pa 5070 2830}%
\special{fp}%
\special{sh 1}%
\special{pa 5070 2830}%
\special{pa 5004 2810}%
\special{pa 5018 2830}%
\special{pa 5004 2850}%
\special{pa 5070 2830}%
\special{fp}%
\put(27.7000,-28.5000){\makebox(0,0)[lb]{$\bullet$}}%
%
\special{pn 13}%
\special{pa 3850 3880}%
\special{pa 3844 3848}%
\special{pa 3836 3818}%
\special{pa 3826 3786}%
\special{pa 3816 3758}%
\special{pa 3800 3730}%
\special{pa 3784 3702}%
\special{pa 3764 3678}%
\special{pa 3742 3654}%
\special{pa 3718 3632}%
\special{pa 3692 3610}%
\special{pa 3666 3590}%
\special{pa 3640 3572}%
\special{pa 3612 3554}%
\special{pa 3584 3538}%
\special{pa 3556 3520}%
\special{pa 3530 3504}%
\special{pa 3502 3488}%
\special{pa 3474 3472}%
\special{pa 3448 3454}%
\special{pa 3422 3436}%
\special{pa 3396 3416}%
\special{pa 3372 3396}%
\special{pa 3350 3374}%
\special{pa 3328 3350}%
\special{pa 3308 3324}%
\special{pa 3288 3296}%
\special{pa 3270 3270}%
\special{pa 3252 3242}%
\special{pa 3236 3214}%
\special{pa 3220 3186}%
\special{pa 3206 3158}%
\special{pa 3190 3134}%
\special{pa 3176 3110}%
\special{pa 3164 3086}%
\special{pa 3152 3060}%
\special{pa 3140 3030}%
\special{pa 3132 2996}%
\special{pa 3128 2958}%
\special{pa 3126 2910}%
\special{pa 3126 2860}%
\special{pa 3126 2818}%
\special{pa 3122 2792}%
\special{pa 3112 2796}%
\special{pa 3100 2820}%
\special{sp}%
%
\special{pn 13}%
\special{pa 3850 1730}%
\special{pa 3844 1762}%
\special{pa 3836 1794}%
\special{pa 3826 1824}%
\special{pa 3816 1854}%
\special{pa 3800 1882}%
\special{pa 3784 1908}%
\special{pa 3764 1934}%
\special{pa 3742 1958}%
\special{pa 3718 1980}%
\special{pa 3692 2000}%
\special{pa 3666 2020}%
\special{pa 3640 2040}%
\special{pa 3612 2056}%
\special{pa 3584 2074}%
\special{pa 3556 2090}%
\special{pa 3530 2106}%
\special{pa 3502 2122}%
\special{pa 3474 2140}%
\special{pa 3448 2156}%
\special{pa 3422 2176}%
\special{pa 3396 2194}%
\special{pa 3372 2216}%
\special{pa 3350 2238}%
\special{pa 3328 2262}%
\special{pa 3308 2288}%
\special{pa 3288 2314}%
\special{pa 3270 2342}%
\special{pa 3252 2370}%
\special{pa 3236 2398}%
\special{pa 3220 2426}%
\special{pa 3206 2452}%
\special{pa 3190 2478}%
\special{pa 3176 2502}%
\special{pa 3164 2526}%
\special{pa 3152 2552}%
\special{pa 3140 2580}%
\special{pa 3132 2614}%
\special{pa 3128 2654}%
\special{pa 3126 2702}%
\special{pa 3126 2752}%
\special{pa 3126 2794}%
\special{pa 3122 2818}%
\special{pa 3112 2816}%
\special{pa 3100 2790}%
\special{sp}%
%
\special{pn 8}%
\special{pa 3310 2130}%
\special{pa 1980 3940}%
\special{fp}%
\special{sh 1}%
\special{pa 1980 3940}%
\special{pa 2036 3898}%
\special{pa 2012 3898}%
\special{pa 2004 3874}%
\special{pa 1980 3940}%
\special{fp}%
\end{picture}%

\vspace{4mm}

Let $V^{(0)} = V$ and $V^{(1)}$ be the Zariski closure
of $V^{(0)} \setminus MV^{(0)}$.
Then we have $V^{(1)} = (V^{(0)} \setminus MV^{(0)}) \cup
\{ (0,0,0) \}$ and $V = MV^{(0)} \cup MV^{(1)}$.
Therefore $V^{(0)} \supset V^{(1)}$ is the canonical filtration
of $V$.
We consider the section $W$ of $V$ by the plane $\{ t = 0 \}$.
Let $W^{(0)} = W$ and $W^{(1)} = V^{(1)} \cap \{ t = 0 \}$.
Then $W^{(1)} = \{ (0,0,0), (0,1,0) \}$ and 
$W^{(0)} \setminus MW^{(0)} = \{ (0,1,0) \}$ is an algebraic
subset of $W$.
Therefore $W^{(0)} \supset W^{(1)}$ is not the canonical
filtration but a compatible one of $W$.

Next let $V^{(0)} = V$ and $V^{(1)}$ be the Nash closure
of $V^{(0)} \setminus MV^{(0)}$.
Then we have $V^{(1)} = V^{(0)} \setminus MV^{(0)}$ 
and $V = MV^{(0)} \cup MV^{(1)}$.
Therefore $V^{(0)} \supset V^{(1)}$ is the Nash canonical filtration
of $V$.
Let $W = V \cap \{ t = 0 \}$ and $W^{(1)} = V^{(1)} \cap \{ t = 0 \}$.
In this case, $W^{(0)} = W \supset W^{(1)}$ is still the Nash canonical
filtration of $W$.
\end{example}


\bigskip
\section{Nash Isotopy Lemma for finiteness property.}
\label{nashisotopylemma}
\medskip

We first recall the Nash Isotopy Lemma proved in \cite{fukuikoikeshiota}.
Let $M \subset \R^m$ be a Nash manifold possibly with boundary,
and let $N_1, \cdots , N_q$ be Nash submanifolds of $M$
possibly with boundary which together with $N_0 = \partial M$
have normal crossings.
Assume that $\partial N_i \subset N_0$, $i = 1, \cdots , q$.
Then we have

\begin{thm}\label{nashisotopy} (\cite{fukuikoikeshiota}).
Let $\varpi : M \to \R^p$, $p > 0$, be a proper onto
Nash submersion such that for every
$0 \le i_1 < \cdots < i_s \le q$,
$\varpi |_{N_{i_1} \cap \cdots \cap N_{i_s}} :
N_{i_1} \cap \cdots \cap N_{i_s} \to \R^p$ is a proper onto submersion
$($unless $N_{i_1} \cap \cdots \cap N_{i_s} = \emptyset )$.
Then there exists a Nash diffeomorphism
$$
\varphi = (\varphi^{\prime},\varpi ) : (M; N_1, \cdots , N_q)
\to (M^{*}; N_1^{*}, \cdots , N_q^{*}) \times \R^p
$$
such that $\varphi |_{M^{*}} =$ id, where $Z^{*}$ denotes
$Z \cap \varpi^{-1}(0)$ for a subset $Z$ of $M$.

Furthermore, if previously given are Nash diffeomorphisms
$\varphi_{i_j} : N_{i_J} \to N_{i_j}^{*} \times \R^p$,
$0 \le i_1 < \cdots < i_a \le q$, such that
$\varpi \circ \varphi_{i_j}^{-1}$ is the natural projection,
and $\varphi_{i_s} = \varphi_{i_t}$ on $N_{i_s} \cap N_{i_t}$,
then we  can choose a Nash diffeomorphism $\varphi$ which
satisfies $\varphi |_{N_{i_j}} = \varphi_{i_j}$,
$j = 1, \cdots a$.
\end{thm}

\begin{rem}\label{remark21}
In the theorem above we can replace $\R^p$ by a Nash open simplex.
\end{rem}

Using the Nash Isotopy Lemma above, 
we prove finiteness theorems in \cite{fukuikoikeshiota, koike1} 
on the existence of Nash trivial simultaneous resolution and 
on Blow-Nash triviality for a family of zero-sets of Nash mappings 
defined over a compact Nash manifold or for a family of Nash set-germs.
In this section we show a kind of modified Isotopy Lemma 
without the assumption of properness which is applicable to
the non-compact case.

Let $M \subset \R^m$ be a Nash manifold, and let $N_1, \cdots , N_q$ 
be Nash submanifolds of $M$ which have normal crossings.
For our purpose, we may assume that $M$, $N_1, \cdots , N_q$ are Nash 
manifolds without boundary and $N_1, \cdots , N_q$ are closed in $M$.
Then we have

\vspace{3mm}

\noindent {\bf Theorem I.} {\em Let $\varpi : M \to \R^p$, $p > 0$, 
be an onto Nash submersion such that 
for every $1 \le i_1 < \cdots < i_s \le q$,
$\varpi |_{N_{i_1} \cap \cdots \cap N_{i_s}} :
N_{i_1} \cap \cdots \cap N_{i_s} \to \R^p$ is an onto submersion
$($unless $N_{i_1} \cap \cdots \cap N_{i_s} = \emptyset )$.
Then there is a finite partition of $\R^p$ into Nash
open simplices $Q_j$, $j = 1, \cdots , b$, and for any $j$,
there exists a Nash diffeomorphism
$$
\varphi_j = (\varphi_j^{\prime},\varpi_j ) : (M_j; N_{1,j}, \cdots , 
N_{q,j}) \to (M_j^{*}; N_{1,j}^{*}, \cdots , N_{q,j}^{*}) \times Q_j
$$
such that} $\varphi_j |_{M_j^{*}} =$ id
{\em where $M_j = \varpi^{-1}(Q_j)$, $\varpi_j = \varpi |_{M_j}$,
$N_{1.j} = N_1 \cap \varpi^{-1}(Q_j), \cdots , 
N_{q.j} = N_q \cap \varpi^{-1}(Q_j)$, and $Z^{*}$ denotes
$Z \cap \varpi_j^{-1}(P)$ for a subset $Z$ of $M_j$
and some point $P \in Q_j$}.

\vspace{3mm}

Thanks to the Semialgebraic Triangulation Theorem, to see our
modified Nash Isotopy Lemma, it suffices to show the following
weak form:

\begin{prop}\label{reduction}
Under the same assumption as Theorem I, there is a finite partition
$\R^p = Q_1 \cup \cdots \cup Q_b \cup R$ which satisfies the following
conditions:
 
 (1) Each $Q_j$ is a Nash open simplex of $\dim p$, and $R$ is 
a semialgebraic subset of $\R^p$ of dimension less than $p$.
 
 (2) For each $j$, $1 \le j \le b$,
there exists a Nash diffeomorphism
$$
\varphi_j = (\varphi_j^{\prime},\varpi_j ) : (M_j; N_{1,j}, \cdots , 
N_{q,j}) \to (M_j^{*}; N_{1,j}^{*}, \cdots , N_{q,j}^{*}) \times Q_j
$$
such that $\varphi_j |_{M_j^{*}} =$ id.
\end{prop}

\begin{proof}
As a tacit understanding, we are assuming that $\dim M \ge p$.
In the case where $\dim M = p$, there is a finite partition
$\R^p = Q_1 \cup \cdots \cup Q_b$ into Nash open simplices such that
for $1 \le j \le b$,
$$
\varpi |_{\varpi^{-1}(Q_j)} : \varpi^{-1}(Q_j) \to Q_j
$$
is Nash trivial.
Incidentally, each restricted mapping is proper.
Therefore, from the beginning we may assume that $\dim M > p$.

Let $M \subset \R^{m^{\prime}}$ (where $\R^m \subset \R^{m^{\prime}}$)
so that $\overline{M} \setminus M$ is a point and $M$ is bounded.
Here $\overline{M}$ denotes the closure of $M$ in $\R^{m^{\prime}}$.
Replace $M$ and $\varpi$ by $M^{\prime} =$ graph $\varpi$ and 
$\varpi^{\prime} : M^{\prime} \to \R^p$ the projection, respectively.
Then $\varpi^{\prime}$ is extended to a semialgebraic mapping
$\overline{\varpi^{\prime}} : \overline{M^{\prime}} \to \R^p$,
and for each $y \in \R^p$, $(\overline{\varpi^{\prime}})^{-1}(y)
\setminus {\varpi^{\prime}}^{-1}(y)$ is empty or a point.
In this paper a semialgebraic mapping means a continuous mapping 
whose graph is semialgebraic.
There are a closed semialgebraic subset $X$ of 
$\overline{M^{\prime}} \setminus M^{\prime}$ of $\dim < p$
and a finite partition $\R^p = Q_1 \cup \cdots \cup Q_b \cup R$
where $R = \overline{\varpi^{\prime}}(X)$,
which satisfy the following conditions:

(1) Each $Q_j$ is a Nash open simplex of $\dim p$.

(2) If $\hat{M}_j = (\overline{M^{\prime}} \setminus M^{\prime})
\cap (\overline{\varpi^{\prime}})^{-1}(Q_j)$ is not empty,
then it is a Nash manifold and $\overline{\varpi^{\prime}} |_{\hat{M}_j}
: \hat{M}_j \to Q_j$ is a Nash diffeomorphism.
In addition, if for $1 \le i_1 < \cdots < i_a \le q$,
$\overline{N_{i_1}^{\prime} \cap \cdots \cap N_{i_a}^{\prime}}
\cap \hat{M}_j \ne \emptyset$, then $\hat{M}_j \subset
\overline{N_{i_1}^{\prime} \cap \cdots \cap N_{i_a}^{\prime}}$,
where each $N_i^{\prime}$ is defined similarly to $M^{\prime}$.

In the case where $\hat{M}_j = \emptyset$, by Theorem \ref{nashisotopy},
there exists a Nash diffeomorphism
$$
\varphi_j = (\varphi_j^{\prime},\varpi_j ) : (M_j; N_{1,j}, \cdots , 
N_{q,j}) \to (M_j^{*}; N_{1,j}^{*}, \cdots , N_{q,j}^{*}) \times Q_j
$$
such that $\varphi_j |_{M_j^{*}} =$ id.
Therefore we assume $\hat{M}_j \ne \emptyset$ after this.
Since $Q_j$ is a Nash open simplex of $\dim p$, it is Nash diffeomorphic 
to $\R^p$. 
For simplicity, we assume the following:

(1) We regard $Q_j$ as $\R^p$.

(2) $\overline{M^{\prime}} \setminus M^{\prime}$ is a Nash manifold, and
$\overline{\varpi^{\prime}} |_{\overline{M^{\prime}} \setminus M^{\prime}} :
\overline{M^{\prime}} \setminus M^{\prime} \to \R^p$ is a Nash diffeomorphism.

(3) For $1 \le i_1 < \cdots < i_a \le q$, if
$\overline{N_{i_1}^{\prime} \cap \cdots \cap N_{i_a}^{\prime}}
\cap (\overline{M^{\prime}} \setminus M^{\prime}) \ne \emptyset$, 
then $\overline{M^{\prime}} \setminus M^{\prime} \subset
\overline{N_{i_1}^{\prime} \cap \cdots \cap N_{i_a}^{\prime}}$.

Set $\alpha (x) =$ dis $(x,(\overline{\varpi^{\prime}}|_{\overline
{M^{\prime}} \setminus M^{\prime}})^{-1}(\overline{\varpi^{\prime}}(x)))$
for $x \in \overline{M^{\prime}}$. 
Then there is a closed semialgebraic neighbourhood $W$ of
$\overline{M^{\prime}} \setminus M^{\prime}$ in $\overline{M^{\prime}}$
with the following properties:

(1) $\alpha$ is $C^{\infty}$ smooth on $U = M^{\prime} \cap W$.

(2) For each $y \in \R^p$, $\alpha |_{U \cap {\varpi^{\prime}}^{-1}(y)}$
is $C^{\infty}$ regular, that is $(\varpi^{\prime},\alpha )$ is
$C^{\infty}$ regular on $U$.

(3) Let $1 \le i_1 < \cdots < i_a \le q$.
In the case where $\overline{N_{i_1}^{\prime} \cap \cdots \cap 
N_{i_a}^{\prime}} \cap (\overline{M^{\prime}} \setminus M^{\prime}) 
= \emptyset$, we have
$(N_{i_1}^{\prime} \cap \cdots \cap N_{i_a}^{\prime}) \cap U = \emptyset$.
In the case where $\overline{N_{i_1}^{\prime} \cap \cdots \cap 
N_{i_a}^{\prime}} \cap (\overline{M^{\prime}} \setminus M^{\prime}) 
\ne \emptyset$, $\alpha |_{U \cap (N_{i_1}^{\prime} \cap \cdots \cap 
N_{i_a}^{\prime}) \cap {\varpi^{\prime}}^{-1}(y)}$ 
is $C^{\infty}$ regular for each $y \in \R^p$.

Multiply $\alpha$ with some large (positive) Nash function in
variables of $\R^p$.
Then we can assume that $U = \alpha^{-1}((0,1])$.
Since $(\varpi^{\prime},\alpha )|_U$ is proper onto $\R^p \times (0,1]$,
$(\varpi^{\prime},\alpha )|_{U \cap (N_{i_1}^{\prime} \cap \cdots \cap 
N_{i_a}^{\prime})}$'s (where
$\overline{N_{i_1}^{\prime} \cap \cdots \cap N_{i_a}^{\prime}} 
\cap (\overline{M^{\prime}} \setminus M^{\prime}) \ne \emptyset$)
are also proper onto $\R^p \times (0,1]$.
Therefore, by Theorem \ref{nashisotopy},
$(\varpi^{\prime},\alpha )|_{(U;N_1^{\prime} \cap U, \cdots ,
N_q^{\prime} \cap U)}$ is Nash trivial over $\R^p \times (0,1]$.
Hence $\varpi^{\prime}|_{(U;N_1^{\prime} \cap U, \cdots ,
N_q^{\prime} \cap U)}$ is Nash trivial over $\R^p$.
On the other hand, we see that $\varpi^{\prime}|_{(M^{\prime} \setminus 
Int U;N_1^{\prime} \cap (M^{\prime} \setminus Int U), \cdots ,
N_q^{\prime} \cap (M^{\prime} \setminus Int U))}$ 
is also Nash trivial over $\R^p$.
Recall the proof of Theorem \ref{nashisotopy}.
Then it turns out that we can construct the Nash triviality of
$\varpi^{\prime}|_{(M^{\prime} \setminus Int U;
N_1^{\prime} \cap (M^{\prime} \setminus Int U), \cdots ,
N_q^{\prime} \cap (M^{\prime} \setminus Int U))}$ over $\R^p$
which follows from the Nash one of $\varpi^{\prime}|_{(U;N_1^{\prime} \cap U, 
\cdots ,N_q^{\prime} \cap U)}$ over $\R^p$
in the $C^1$ Nash category.
Thus, using a similar argument to the proof of Theorem \ref{nashisotopy}
based on Shiota's Nash approximation theorem, we obtain a Nash triviality of 
$\varpi^{\prime} : (M^{\prime};N_1^{\prime}, \cdots ,N_q^{\prime})
\to \R^p$.

This completes the proof of the proposition.
\end{proof}

Using Theorem I we can improve some finiteness theorems
proved in \cite{koike1}.
Let $N$ be a Nash manifold, and let $J$ be a semialgebraic 
set in some Euclidean space.
Let $f_t : N \to \R^k\ (t \in J)$ be a Nash mapping.
Define $F : N \times J \to \R^k$ by $F(x;t) = f_t(x)$.
Assume that $F$ is a Nash mapping. Set
$$
K = \{t \in J\ |\ f_t^{-1}(0) \cap S(f_t) \ is \ isolated \}.
$$
As seen in \cite{koike1}, $K$ is a semialgebraic subset of $J$.
By a similar argument to \cite{koike1} with Theorem I,
we can show the following:

\vspace{3mm}

\noindent {\bf Theorem II.} 
{\em There exists a finite partition 
$$
J = Q_1 \cup \cdots \cup Q_s \cup Q_{s+1} \cup \cdots \cup Q_u
$$
with $K = Q_1 \cup \cdots \cup Q_s$ 
and $J - K = Q_{s+1} \cup \cdots \cup Q_u$
which satisfies the following conditions:

(1) Each $Q_i$ is a Nash open simplex.

(2) For each $i$, there is a Nash trivial simultaneous resolution
$\Pi_i : \mathcal{M}_i \to N \times Q_i$ of $F_{Q_i}^{-1}(0)$ 
in $N \times Q_i$ over $Q_i$.

In particular, for $1 \le i \le s$, this Nash trivialisation
induces a semialgebraic trivialisation of $F_{Q_i}^{-1}(0)$
in $N \times Q_i$ over $Q_i$.
Therefore $(N \times Q_i,F_{Q_i}^{-1}(0))$
admits a $\Pi_i$-Blow Nash trivialisation along $Q_i$.}


\bigskip
\section{Programme to show finiteness on Blow-semialgebraic triviality.}
\label{programme}
\medskip

Let $N$ be a Nash manifold of dimension $n$,
and let $J$ be a semialgebraic set in some Euclidean space.
Let $f_t : N \to \R^k$ ($t \in J$) be a Nash mapping.
Assume that $F : N \times J \to \R^k$ is a Nash mapping.

In this section we give a programme to show a finiteness
theorem on Blow-semialgebraic triviality for $(N \times J,F^{-1}(0))$
under some assumptions on $N$ and $F$.
As mentioned in the introduction, we divide our programme
into 8 processes.
Processes I, II and III always work for any Nash manifold $N$
and Nash mapping $F$.
On the other hand, some properness condition
is required in Processes IV - VI,
which can be applied to the compact case, namely,
the case where $N$ is a compact Nash manifold.
We terminate our programme in the compact case at Process VI.
To carry out our programme in the non-compact case, 
some device is necessary.
In fact, we use two reduction methods in our programme.
One is a reduction from the non-compact case to the compact case
for a nonsingular algebraic variety $N$,
and another is from the Nash case to the algebraic case
for a non-compact Nash manifold $N$.
We describe the first and second methods as Processes VII
and VIII, respectively.

\vspace{3mm}

\noindent {\bf Process I. Constancy of dimensions.}
Denote
$$
SF^{-1}_Q(0) = \{ (x,t) \in N \times Q \ | \
x \in f_t^{-1}(0) \cap S(f_t) \}
$$
for $Q \subset J$.
Then, by Theorem \ref{hardt}, there is a finite partition
of $J$ into semialgebraic sets $Q_i$ such that
$F_{Q_i}^{-1}(0)$ and $SF_{Q_i}^{-1}(0)$ are
semialgebraically trivial over each $Q_i$.
Therefore, when we consider our finiteness problem,
we may assume from the beginning that
$\dim f_t^{-1}(0)$ and $\dim f_t^{-1}(0) \cap S(f_t)$
are constant over $J$.
In addition, we have already known the following:

(i) In the case where $\dim f_t^{-1}(0) \cap S(f_t) = -1$
i.e. $f_t^{-1}(0) \cap S(f_t) = \emptyset$,
a finiteness theorem holds on Nash triviality (\cite{costeshiota}).

(ii) In the case where $\dim f_t^{-1}(0) \cap S(f_t) = 0$
i.e. in the case of isolated singularities,
a finiteness theorem holds on Blow-Nash triviality
(Theorem II).

After this, we assume that $\dim f_t^{-1}(0)$ is constant
and $\dim f_t^{-1}(0) \cap S(f_t) \ge 1$ over $J$.

\vspace{3mm}

\noindent {\bf Process II. Finiteness on the existence
of Nash trivial simultaneous resolution.}
By Theorem II, there exists a finite partition of
$J = Q_1 \cup \cdots \cup Q_u$ such that for each $i$,

(1) $Q_i$ is a Nash open simplex, and

(2) there is a Nash trivial simultaneous resolution
$\Pi_i : \mathcal{M}_i \to N \times Q_i$ of $F_{Q_i}^{-1}(0)$
in $N \times Q_i$ over $Q_i$.

In order to show this finiteness theorem,
we used the desingularisation theorem of Hironaka or
Bierstone-Milman (Theorem \ref{desingularisation}).
Therefore, for each $i$
\begin{equation}\label{III-1}
\Pi_i(S\Pi_i) \subset Sing F_{Q_i}^{-1}(0).
\end{equation}

Let $q : N \times Q_i \to Q_i$ be the canonical projection.
For $t \in Q_i$, we set $M_t = \mathcal{M}_i \cap (q \circ \Pi_i)^{-1}(t)$
and $N_t = N \times \{ t \}$.
Define a mapping $\pi_t : M_t \to N_t$ by
$\pi_t =\Pi_i|_{M_t} : M_t \to N_t$.
Then it follows from \eqref{III-1} and the Nash triviality that

\vspace{3mm}

\centerline{$\pi_t(S\pi_t) \subset Sing f_t^{-1}(0)$
\ \ for any $t \in Q_i$.}

\vspace{3mm}

\noindent Hence $\dim \pi_t(S\pi_t) \le \dim Sing f_t^{-1}(0)
< \dim f_t^{-1}(0) = K_i$ (constant), namely,

\vspace{3mm}

\centerline{$\dim \pi_t(S\pi_t) \le K_i - 1$ \ \  for any $t \in Q_i$.}

\vspace{3mm}

\noindent {\bf Process III. Fukuda's lemma on a stratified mapping.}
By T. Fukuda \cite{fukuda1}, any Nash mapping $f : M \to N$
between Nash manifolds can be stratified.

\begin{lem}\label{fukuda}
Given semialgebraic subsets $A_1, \cdots , A_a$ of $M$ 
and semialgebraic subsets $B_1, \cdots , B_b$ of $N$,
there exist finite $C^{\omega}$-Nash Whitney stratifications
$\mathcal{S}(M)$ of $M$ compatible with $A_1, \cdots , A_a$
and $\mathcal{S}(N)$ of $N$ compatible with $B_1, \cdots , B_b$
such that for any $X \in \mathcal{S}(M)$, there is a stratum
$U \in \mathcal{S}(N)$ such that the restriction
$f|_X : X \to U$ is a Nash submersion.
\end{lem}

\begin{rem}\label{remark31}
 Taking substratifications if necessary, we may assume that
$f|_X : X \to U$ is surjective.
\end{rem}

Set $W_i = F_{Q_i}^{-1}(0)$.
Then it follows from Lemma \ref{fukuda} and Theorem \ref{lojasiewicz}
that taking a finite partition of $Q_i$ if necessary,
there are finite $C^{\omega}$-Nash Whitney stratifications
$\mathcal{S}(\mathcal{M}_i)$ of $\mathcal{M}_i$
compatible with $W_i^{\prime},\ \mathcal{D}_1, \cdots , \mathcal{D}_d$
and their intersections, 
and $\mathcal{S}(N \times Q_i)$ of $N \times Q_i$
compatible with their images by $\Pi_i$ so that
$\Pi_i : (\mathcal{M}_i,\mathcal{S}(\mathcal{M}_i)) \to 
(N \times Q_i,\mathcal{S}(N \times Q_i))$ and
$q : (N \times Q_i,\mathcal{S}(N \times Q_i)) \to
(Q_i,\{ Q_i \} )$ are stratified mappings.
Here $W_i^{\prime}$ is the strict transform of $W_i$ by $\Pi_i$
and $\mathcal{D}_1, \cdots , \mathcal{D}_d$ are the exceptional divisors.

Set

\vspace{3mm}

\centerline{$\mathcal{S}(\mathcal{D}_1 \cup \cdots \cup \mathcal{D}_d)
= \{ X \in \mathcal{S}(\mathcal{M}_i)\ |\
X \subset \mathcal{D}_1 \cup \cdots \cup \mathcal{D}_d \}$ ,}

\vspace{2mm}

\centerline{$\mathcal{S}(\Pi_i(\mathcal{D}_1 \cup \cdots \cup \mathcal{D}_d))
= \{ U \in \mathcal{S}(N \times Q_i)\ |\
U \subset \Pi_i(\mathcal{D}_1 \cup \cdots \cup \mathcal{D}_d) \}$ .}

\vspace{3mm}

\noindent We can assume that any stratum in 
$\mathcal{S}(\mathcal{M}_i)$ and $\mathcal{S}(N \times Q_i)$
is connected.

\vspace{3mm}

\noindent {\bf Process IV. Semialgebraic triviality of 
$\Pi_i|_{\mathcal{D}_1 \cup \cdots \cup \mathcal{D}_d}$.}
We first recall the semialgebraic version of Thom's 2nd Isotopy Lemma
proved by M. Shiota \cite{shiota2}.

Let $A \subset \R^m,\ B \subset \R^r$
be semialgebraic sets, let $I$ be a $C^2$-Nash open simplex,
and let $f : A \to B$ and $q : B \to I$ be proper $C^2$-Nash mappings.

\begin{lem}\label{2ndisotopy}
Suppose that $A$ and $B$ admit finite $C^2$-Nash Whitney stratifications
$\mathcal{S}(A)$ and $\mathcal{S}(B)$ respectively such that
$f : (A,\mathcal{S}(A)) \to (B,\mathcal{S}(B))$ is a Thom mapping
and $q : (B,\mathcal{S}(B)) \to (I,\{ I \} )$ is a stratified mapping.
Then the stratified mapping $f$ is 
semialgebraically trivial over $I$,
namely, there are semialgebraic homeomorphisms
$H : (q \circ f)^{-1}(P_0) \times I \to A$ and
$h : q^{-1}(P_0) \times I \to B$, for some $P_0 \in I$,
preserving the natural stratifications such that
$h^{-1} \circ f \circ H = f|_{(q \circ f)^{-1}(P_0)} \times id_{I}$
and $q \circ h : q^{-1}(P_0) \times I \to I$ 
is the canonical projection.
\end{lem}

\begin{rem}\label{remark32} (1) The Whitney regularity of 
adjacent strata of $\mathcal{S}(A)$ mapped into different strata 
of $\mathcal{S}(B)$ is not necessary to show the triviality
of $f$ over $I$.
In \cite{sabbah2}, C. Sabbah calls a stratified mapping 
$(f,\mathcal{S} = \{ \mathcal{S}(A),\mathcal{S}(B) \} )$
{\em sans \'eclatement} if $(f,\mathcal{S})$ satisfies
the assumptions of Thom's 2nd Isotopy Lemma
without the Whitney regularity of such adjacent strata.

(2) It is well-known that in Thom's 2nd Isotopy Lemma
we can weaken the assumptions of the Whitney regularity 
and the Thom regularity more.
The Lemma is shown under the assumption
on the existence of controlled $C^2$ tube systems 
with some geometric contents.
For instance, see Complement of Theorem II.$6.1^{\prime}$
in \cite{shiota2}.
\end{rem}

Let $A$, $B$, $I$, $f : A \to B$ and $q : B \to I$ 
be the same as above.
Assume that $A$ and $B$ admit finite $C^2$-Nash Whitney stratifications 
$\mathcal{S}(A)$ and $\mathcal{S}(B)$, respectively.
Assume further that $f : (A,\mathcal{S}(A)) \to (B,\mathcal{S}(B))$ 
and $q : (B,\mathcal{S}(B)) \to (I,\{ I \} )$
are stratified mappings.
For $t \in I$, let $A_t$ and $B_t$ denote
$(q \circ f)^{-1}(t) \cap A$ and $q^{-1}(t) \cap B$ 
respectively, and let

\vspace{3mm}

\qquad $\mathcal{S}(A)_t = (q \circ f)^{-1}(t) \cap \mathcal{S}(A)$
and $\mathcal{S}(B)_t = q^{-1}(t) \cap \mathcal{S}(B).$

\vspace{3mm}

\noindent Note that for any $t \in I$, $\mathcal{S}(A)_t$ and 
$\mathcal{S}(B)_t$ are stratifications of $A_t$ and $B_t$ respectively,
and $f_t = f|_{A_t} : (A_t,\mathcal{S}(A)_t) \to (B_t,\mathcal{S}(B)_t)$ 
is a stratified mapping .
Then we can show the following modified Thom's 2nd Isotopy Lemma 
in the semialgebraic category using the arguments 
based on the proof of the above 2nd Isotopy Lemma.

\begin{lem}\label{weakisotopy}(\cite{koike1})
Suppose that for any $t \in I$, the stratified mapping
$f_t : (A_t,\mathcal{S}(A)_t) \to (B_t,\mathcal{S}(B)_t)$ 
is $(a_{f_t})$-regular.
Then, subdividing $I$ into finitely many $C^2$-Nash manifolds
if necessary, the stratified mapping
$f : (A,\mathcal{S}(A)) \to (B,\mathcal{S}(B))$ 
is semialgebraically trivial over $I$.
\end{lem}

\noindent Throughout this paper, we use the notations
$A_t$, $B_t$, $\mathcal{S}(A)_t$, $\mathcal{S}(B)_t$
and $f_t : (A_t,\mathcal{S}(A)_t) \to (B_t,\mathcal{S}(B)_t)$
in the above meaning for stratified mappings 
$f : (A,\mathcal{S}(A)) \to (B,\mathcal{S}(B))$ 
and $q : (B,\mathcal{S}(B)) \to (I,\{ I \} )$.

Let us go back to our process.
>From Process IV to Process VI, we assume that $N$ is compact
except the statements concerning Isotopy Lemmas.
Let $\Pi_i : (\mathcal{M}_i,\mathcal{S}(\mathcal{M}_i)) \to
(N \times Q_i,\mathcal{S}(N \times Q_i))$
be a stratified mapping given in Process III.
We consider the restriction of the stratified mapping $\Pi_i$ to
$\mathcal{D}_1 \cup \cdots \cup \mathcal{D}_d$
as the above $f$.

\vspace{2mm}

{\em IV-A.} We first consider the case where the following
assumption is satisfied for $Q_i$.

\vspace{2mm}

\noindent {\bf Assumption A.} {\em For any} $t \in Q_i$,
{\em the stratified mapping}
$$
\pi_t|_{(\mathcal{D}_1 \cup \cdots \cup \mathcal{D}_d)_t} :
((\mathcal{D}_1 \cup \cdots \cup \mathcal{D}_d)_t,
\mathcal{S}(\mathcal{D}_1 \cup \cdots \cup \mathcal{D}_d)_t) \to
(\Pi_i(\mathcal{D}_1 \cup \cdots \cup \mathcal{D}_d)_t,
\mathcal{S}(\Pi_i(\mathcal{D}_1 \cup \cdots \cup \mathcal{D}_d))_t)
$$
{\em is} ($a_{\pi_t}$)-{\em regular.}

\vspace{3mm}

By the compactness of $N$, it follows from Lemma \ref{weakisotopy}
that taking a finite partition 
of $Q_i$ if necessary, the stratified mapping
$$
\Pi_i|_{\mathcal{D}_1 \cup \cdots \cup \mathcal{D}_d} :
(\mathcal{D}_1 \cup \cdots \cup \mathcal{D}_d,
\mathcal{S}(\mathcal{D}_1 \cup \cdots \cup \mathcal{D}_d)) \to
(\Pi_i(\mathcal{D}_1 \cup \cdots \cup \mathcal{D}_d),
\mathcal{S}(\Pi_i(\mathcal{D}_1 \cup \cdots \cup \mathcal{D}_d)))$$
is semialgebraically trivial over $Q_i$.

\vspace{3mm}

\noindent {\em IV-B.} We next consider the case where
the ($a_{\pi_t}$)-regularity is not satisfied for $Q_i$.
Then we pose the following:

\vspace{3mm}

\noindent {\bf Assumption B.} {\em There exist a thin semialgebraic
subset $R_i$ of $Q_i$ and a finite partition
$Q_i \setminus R_i = Q_{i,1} \cup \cdots \cup Q_{i,v}$
such that each $Q_{i,j}$ is a Nash open simplex, 
and the stratified mapping}
$$
\Pi_i|_{\mathcal{D}_1 \cup \cdots \cup \mathcal{D}_d} :
(\mathcal{D}_1 \cup \cdots \cup \mathcal{D}_d,
\mathcal{S}(\mathcal{D}_1 \cup \cdots \cup \mathcal{D}_d))
\to (\Pi_i(\mathcal{D}_1 \cup \cdots \cup \mathcal{D}_d),
\mathcal{S}(\Pi_i(\mathcal{D}_1 \cup \cdots \cup \mathcal{D}_d)))
$$
{\em is semialgebraically trivial over} $Q_{i,j}$, $1 \le j \le v$.

\vspace{3mm}

\noindent {\bf Process V. Semialgebraic triviality of $\Pi_i$.}
We recall the semialgebraic version of Thom's 1st Isotopy Lemma.

\begin{lem}\label{1stisotopy} (\cite{shiota2})
Let $A$ be a semialgebraic subset of $\R^m$,
let $I$ be a $C^2$-Nash open simplex,
and let $f : A \to I$ be a proper $C^2$-Nash mapping.
Suppose that $A$ admits a finite $C^2$-Nash Whitney stratification
$\mathcal{S}(A)$ such that for any stratum $X \in \mathcal{S}(A),\
f|_X : X \to I$ is a $C^2$ submersion onto $I$.
Then the stratified set $(A,\mathcal{S}(A))$ is
semialgebraically trivial over $I$.
\end{lem}

Let us consider the situation that by Process IV we can have  
a semialgebraic trivialisation of the stratified mapping
$\Pi_i|_{\mathcal{D}_1 \cup \cdots \cup \mathcal{D}_d}$
over $Q_i$ or $Q_{i,j}$.
Here we explain only how we obtain the semialgebraic triviality
of $\Pi_i$ over $Q_i$, since the argument for $Q_{i,j}$ is the same.
The semialgebraic trivialisation of
$\Pi_i|_{\mathcal{D}_1 \cup \cdots \cup \mathcal{D}_d}$ over $Q_i$
gives a semialgebraic one of the stratified mapping
$$
q \circ \Pi_i|_{\mathcal{D}_1 \cup \cdots \cup \mathcal{D}_d} :
(\mathcal{D}_1 \cup \cdots \cup \mathcal{D}_d,
\mathcal{S}(\mathcal{D}_1 \cup \cdots \cup \mathcal{D}_d))
\to (Q_i,\{ Q_i \} ).
$$
This stratified mapping is the restriction to 
$\mathcal{D}_1 \cup \cdots \cup \mathcal{D}_d$
of the stratified mapping
$$
q \circ \Pi_i : (\mathcal{M}_i,\mathcal{S}(\mathcal{M}_i)) 
\to (Q_i,\{ Q_i \} ).
$$

The stratified mapping $q \circ \Pi_i$ 
satisfies the hypotheses of lemma \ref{1stisotopy}.
Therefore there is a semialgebraic trivialisation of
$(\mathcal{M}_i,\mathcal{S}(\mathcal{M}_i))$ over $Q_i$
extending the above trivialisation of
$(\mathcal{D}_1 \cup \cdots \cup \mathcal{D}_d,
\mathcal{S}(\mathcal{D}_1 \cup \cdots \cup \mathcal{D}_d))$
over $Q_i$.
Since $\Pi_i$ is a Nash isomorphism outside
$\mathcal{D}_1 \cup \cdots \cup \mathcal{D}_d$,
the extended semialgebraic trivialisation induces
a semialgebraic one of the stratified mapping
$$
\Pi_i : (\mathcal{M}_i,\mathcal{S}(\mathcal{M}_i)) \to
(N \times Q_i,\mathcal{S}(N \times Q_i))$$
over $Q_i$.
In this way, we obtain a $\Pi_i$-Blow semialgebraic trivialisation 
of $(N \times Q_i,F_{Q_i}^{-1}(0))$ along $Q_i$.

\vspace{3mm}

\noindent {\bf Process VI. Finiteness on Blow-semialgebraic triviality.}
If Assumption A is satisfied over any $Q_i$ in Process IV,
then finiteness on Blow-semialgebraic triviality for 
$\{ (N,f_t^{-1}(0)) \}_{t \in J}$ follows from Process V.

In the case where not Assumption A but Assumption B is satisfied over
some $Q_i$, let $R_i$ be the thin semialgebraic subset of $Q_i$
removed in Process IV.
We take a finite subdivision of $R_i$ into Nash open simplices $R_{i,j}$'s.
By Process II, taking a finite subdivision of $R_{i,j}$ if necessary,
there is a Nash trivial simultaneous resolution
$\Pi_{i,j} : \mathcal{M}_{i,j} \to N \times R_{i,j}$
of $F_{R_{i,j}}^{-1}(0)$ in $N \times R_{i,j}$.
In addition, we may assume that $\Pi_{i,j}$ and $q$ over
each $R_{i,j}$  have the same properties as $\Pi_i$ and $q$
in Process III, respectively.
Therefore we can advance to Process IV in our programme
with $\{ (N,f_t^{-1}(0)) \}_{t \in R_{i,j}}$.
Since the dimension of each $R_{i,j}$
is less than that of $Q_i$,
we can show finiteness on Blow-semialgebraic triviality
for $\{ (N,f_t^{-1}(0)) \}_{t \in J}$,
if Assumption A or B is always satisfied in our
circulatory programme.

By the arguments mentioned above, we can terminate our programme
in the compact case.

\begin{example}\label{example2}(Theorem IIIa in \cite{koike1})
Let $N$ be a compact Nash manifold of dimension $3$,
and let J be a semialgebraic set in some Euclidean space.
Let $f_t : N \to \R^k$, $t \in J$, be a Nash mapping
such that $\dim f_t^{-1}(0) = 2$ over $J$.
Assume that $F$ is a Nash mapping.
Then we put the programme from Process I to VI into practice
for the family of Nash sets $(N \times J, F^{-1}(0))$.
In Process IV, Assumption A is satisfied.
See \cite{koike1} for the details.
Therefore finiteness on Blow-semialgebraic triviality
for $(N \times J, F^{-1}(0))$ follows from our programme.
\end{example}

\noindent {\bf Process VII. Reduction from the non-compact algebraic case 
to the algebraic compact case.} 
In this process, let $N$ be a nonsingular (affine) algebraic variety
in $\R^m$, let $J$ be a Nash open simplex in $\R^a$,
and let $F : N \times J \to \R^k$ be a Nash mapping.
Assume that $N$ is non-compact.

We consider the projectivisation of $N$ in $\R P^m$,
denoted by $\hat{N}$.
Let us regard $\R^m \subset \R P^m$.
Then $\hat{N}$ may be singular at some points in $\hat{N} \setminus N$.
By the desingularisation theorem of Hironaka,
there is an algebraic desingularisation $\gamma : \mathcal{N} \to \hat{N}$
of $\hat{N}$ whose exceptional set is mapped to $\hat{N} \setminus N$.
Let $V = F^{-1}(0)$ and $\hat{V}$ the Nash closure of $V$ in $\R P^m \times J$.
Note that $\hat{V} \subset \hat{N} \times J$ and $\hat{V} \cap \R^m \times J = V$.
We define a map $\beta : \mathcal{M} \to \hat{N} \times J$ by
$\beta (x;t) = (\gamma (x), t)$, where $\mathcal{M} = \mathcal{N} \times J$.
By construction, $\beta |_{\beta^{-1}(N \times J)} : \beta^{-1}(N \times J) 
\to N \times J$ is a Nash isomorphism.

We denote by $W$ the strict transform of $\hat{V}$ by $\beta$.
Let us consider $W$ as a Nash family of Nash sets
defined over a compact, nonsingular algebraic variety $\mathcal{N}$,
and let $q : \mathcal{M} = \mathcal{N} \times J \to J$ be the
canonical projection.
Then we can apply our programme, Processes I - VI, to this 
$(\mathcal{M},W)$.
If Assumption A or B is always satisfied at Process IV in the circulatory 
programme, there is a finite partition $J$ into
Nash open simplices $Q_i$'s such that $(\mathcal{N} \times Q_i,W_i)$ 
admits a $\Pi_i$-Blow-semialgebraic trivialisation along $Q_i$,
where $W_i = W \cap q^{-1}(Q_i)$ and $\Pi_i : \tilde{\mathcal{M}}_i \to
\mathcal{N} \times Q_i$ is a Nash trivial simultaneous resolution
of $W_i$ in $\mathcal{N} \times Q_i$. 
Let $\tilde{\Pi}_i$ be the restriction of $\beta \circ \Pi_i$ 
to $(\beta \circ \Pi_i)^{-1}(N \times Q_i)$.
At each Process III in the circulatory programme,
we can take the stratifications of $\tilde{\mathcal{M}}_i$ and
$\mathcal{N} \times Q_i$ so that the former stratification is compatible
with $\beta^{-1}(N \times Q_i)$ and 
$\beta^{-1}(\hat{V} \cap N \times Q_i)$, 
and the latter one is compatible with $N \times Q_i$
and $\hat{V} \cap N \times Q_i$.
We are regarding $F_{Q_i}^{-1}(0) = V \cap N \times Q_i$ 
as $\hat{V} \cap N \times Q_i$.
Therefore, by construction, the $\Pi_i$-Blow-semialgebraic trivialisation
of $(\mathcal{N} \times Q_i,W_i)$ induces a
$\tilde{\Pi}_i$-Blow-semialgebraic trivialisation of 
$(N \times Q_i,F_{Q_i}^{-1}(0))$ along $Q_i$.
In this way, we can reduce the algebraic non-compact case
to the compact case.

\vspace{3mm}

\noindent {\bf Process VIII. Reduction from the Nash case to
the nonsingular algebraic case.}
Let $N$ be a Nash manifold, and let $J$ be a Nash open simplex.
Let $F : N \times J \to \R^k$ be a Nash mapping.
Since every Nash manifold is Nash diffeomorphic to a nonsingular
(affine) algebraic variety (M. Shiota \cite{shiota1}),
we may assume that $N$ is a nonsingular algebraic variety.
Therefore we can reduce this Nash case to the above algebraic case. 

\vspace{3mm}

In Example \ref{example2}, let $N$ be a non-compact Nash manifold.
Then we can apply Processes VII and VIII to this case
to show a similar finiteness result to the example.
As a result, we get a finiteness theorem on Blow-semialgebraic triviality
for $(N \times J, F^{-1}(0))$ in the case where
$N$ is a Nash manifold which is not necessarily compact.

\vspace{3mm}

\begin{rem}\label{remark33}
In this section we have described our programme to show finiteness 
on Blow-semialgebraic triviality for a family of Nash sets as embedded
varieties, and in the next section we apply this programme
to show Theorems III, IV.
But in \S 5 we consider the Blow-semialgebraic triviality 
just for a family of the main parts of algebraic sets.
Therefore we have to modify the programme for the proof.
In Process IV we show semialgebraic triviality of
$\Pi_i|_{\mathcal{V}_i^{\prime} \cap (\mathcal{D}_1 \cup \cdots 
\cup \mathcal{D}_d})$ where $\mathcal{V}_i = F_{Q_i}^{-1}(0)$.
As a result, the semialgebraic triviality constructed in Process IV
is extended to $\mathcal{V}_i^{\prime}$ (not $\mathcal{M}_i$) in Process V,
using a similar argument to the above.
Because of some technical reason, Process VI is slightly
modified also in the proof.
In order to show the non-compact algebraic case, we consider
the complexification.
As a result, the argument in Process VIII is quite different
from this section (see \S 7 for the details).
\end{rem}


\bigskip
\section{Applications.}
\label{application}
\medskip

Let us give some applications of the programme
described in the previous section.
We make a remark on the programme that in the compact Nash case, 
Processes I-III, V always work,
and Process VI works if so does Process IV.
As seen in \S 3, when we consider our finiteness problem,
the non-compact Nash case follows from the compact one 
through Processes VII, VIII.

Let $N$ be a Nash manifold of dimension $n$,
and let $J$ be a semialgebraic set in some Euclidean space.
Let $f_t : N \to \R^k$ ($t \in J$) be a Nash mapping.
By Process I, we assume that $\dim f_t^{-1}(0)$
is constant over $J$ and
$\dim f_t^{-1}(0) \cap S(f_t) \ge 1$ for any $t \in J$.

\vspace{3mm}

\noindent {\bf Theorem III.} {\em Suppose that}
$\dim f_t^{-1}(0) = 2$ {\em over} $J$.
{\em Then there exists a finite partition} 
$J = Q_1 \cup \cdots \cup Q_u$ {\em such that for each} $i$,

(1) $Q_i$ {\em is a Nash open simplex, and}

(2) $(N \times Q_i,F_{Q_i}^{-1}(0))$ {\em admits a Blow-semialgebraic
trivialisation along} $Q_i$.

\begin{rem}\label{remark41} 
In this theorem we are considering the case $n \ge 3$.
The case $n = 3$ is discussed in \cite{koike1} and the
previous section (Processes VI - VIII).
\end{rem}

\begin{proof}[Proof of Theorem III] 
By our programme, it suffices to show the compact case.
Therefore we may assume that $N$ is compact. 
Let $\Pi_i : (\mathcal{M}_i,\mathcal{S}(\mathcal{M}_i)) \to
(N \times Q_i,\mathcal{S}(N \times Q_i))$ and
$q : (N \times Q_i,\mathcal{S}(N \times Q_i)) \to
(Q_i,\{ Q_i \} )$ be the stratified mappings given in Process III.
Then $\dim \pi_t(S\pi_t) \le 1$ for any $t \in Q_i$.

Let $X \in \mathcal{S}(\mathcal{D}_1 \cup \cdots \cup \mathcal{D}_d)$ and
$U \in \mathcal{S}(\Pi_i(\mathcal{D}_1 \cup \cdots \cup \mathcal{D}_d))$.
Since the restrictions $q \circ \Pi_i|_X : X \to Q_i$
and $q|_U : U \to Q_i$ are onto submersions,
$\dim X_t$ and $\dim U_t$ are constant over $Q_i$.
Then we have $0 \le \dim U_t \le 1$ for
$U \in \mathcal{S}(\Pi_i(\mathcal{D}_1 \cup \cdots \cup \mathcal{D}_d))$.
Let $X_t, \ Y_t \in 
\mathcal{S}(\mathcal{D}_1 \cup \cdots \cup \mathcal{D}_d)_t$
and $U_t, \ V_t \in
\mathcal{S}(\Pi_i(\mathcal{D}_1 \cup \cdots \cup \mathcal{D}_d))_t$
such that $\overline{X_t} \supset Y_t, \ \pi_t(X_t) = U_t$
and $\pi_t(Y_t) = V_t$.
In the case where $U_t = V_t$, the ($a_{\pi_t}$)-regularity
of the pair ($X_t,Y_t$) follows from the Whitney ($b$)-regularity.
Therefore it suffices to consider the ($a_{\pi_t}$)-regularity
in the case where $\dim U_t = 1$ and $\dim V_t = 0$.
Since $U_t$ and $V_t$ are connected, $V_t$ is a point and 
$U_t$ is Nash diffeomorphic to an open interval or a circle.
But in the case where $U_t$ is diffeomorphic to a circle,
$V_t$ cannot be contained in the closure of $U_t$.
Therefore $U_t$ is diffeomorphic to an open interval.
Taking substratifications of $\mathcal{S}(\mathcal{M}_i)$,
$\mathcal{S}(N \times Q_i)$ and a subtriangulation of
$\{ Q_i \} $ if necessary, we may assume that $\partial 
U_t$ consists of $V_t$ and another point.
Set

\vspace{3mm}

\centerline{$B(X_t,Y_t) = \{ y \in Y_t \ | \ (X_t,Y_t)$
is not $(a_{\pi_t})$-regular at $y$\} }

\vspace{3mm}

\noindent and $B(X,Y) = \bigcup_{t \in Q_i} B(X_t,Y_t)$.
By \cite{fukuda1}, $B(X,Y)$ is a semialgebraic subset
of $\mathcal{M}_i$.
Now we consider the straightening up of $V_t \cup U_t$.
Namely, there are a (semialgebraic) neighbourhood $W$
of $V_t$ in $N$ and a $C^1$ subanalytic diffeomorphism
$\phi : W \to \R^k$ such that $\phi(V_t) = 0$ and
$\phi(V_t \cup U_t) \subset \R \times \{ 0 \}$.
Then we have

\vspace{3mm}

\centerline{$B(X_t,Y_t) = \{ y \in Y_t \ | \ (X_t \cap \pi_t^{-1}(W),Y_t)$ 
is not $(a_{\phi \circ \pi})$-regular at $y$\} .}

\vspace{3mm}

\noindent Using a similar argument to K. Bekka \cite{bekka}
and K. Kurdyka - G. Raby \cite{kurdykaraby},
we see that $B(X_t,Y_t)$ is a thin subset of $Y_t$.
Therefore we can take subdivisions of
$\mathcal{S}(\mathcal{M}_i)$, $\mathcal{S}(N \times Q_i)$ and $\{ Q_i \}$
so that for any $X, \ Y \in 
\mathcal{S}(\mathcal{D}_1 \cup \cdots \cup \mathcal{D}_d)$
with $\overline{X} \supset Y$ and any $t \in Q_i$,
$(X_t,Y_t)$ is $(a_{\pi_t})$-regular.
Since Assumption A in Process IV is satisfied over any $Q_i$,
the statement of the theorem follows.
\end{proof}

In \cite{koike2} we gave a list on finiteness properties
for a family of zero-sets of Nash mappings defined over a compact
Nash manifold $N$ (or for a family of Nash set-germs).
By Theorems II and III, we have got some improvements in it.
We give the updated list below.
Note that in this case the Nash manifold $N$ is not necessarily compact.

\vspace{8mm}

\begin{center}
\setlength{\unitlength}{0.7mm}
\scriptsize
\begin{picture}(188,90)
\put(0,0){\line(1,0){188}}
\put(0,90){\line(1,0){188}}
\put(0,75){\line(1,0){188}}
\put(0,0){\line(0,1){90}}
\put(35,0){\line(0,1){90}}
\put(58,0){\line(0,1){90}}
\put(112,0){\line(0,1){90}}
\put(188,0){\line(0,1){90}}
\put(35,60){\line(1,0){130}}
\put(35,45){\line(1,0){23}}
\put(35,30){\line(1,0){23}}
\put(35,10){\line(1,0){23}}
\put(112,60){\line(1,0){76}}
\put(112,45){\line(1,0){76}}
\put(112,30){\line(1,0){76}}
\put(112,10){\line(1,0){76}}
\put(2,35)
{\begin{minipage}{20mm}
Existence of Nash trivial simultaneous resolution
\end{minipage}}
\put(36,80){\tiny $\dim f_t^{-1}(0)$}
\put(45,66){$0$}
\put(45,50){$1$}
\put(45,35){$2$}
\put(45,18){$3$}
\put(45,3){$\ge4$}
\put(60,83){$f_t^{-1}(0)\cap S(f_t)$ : isolated} 
\put(60,78){(resp. $=\emptyset$)}
\put(60,68){Nash triviality for} 
\put(60,63){$\{(N,f_t^{-1}(0))\}$}
\put(60,34){Blow-Nash triviality for} 
\put(60,29){(resp. Nash triviality for)}
\put(60,24){$\{(N,f_t^{-1}(0))\}$}
\put(115,80){$f_t^{-1}(0)\cap S(f_t)$ : non-isolated}
\put(115,66){This case never happens.}
\put(115,53){Blow-Nash triviality for}
\put(115,48){$\{(N,f_t^{-1}(0))\}$}
\put(115,38){Blow-semialgebraic triviality for}
\put(115,33){$\{(N,f_t^{-1}(0))\}$}
\put(115,23){Main result (Blow-semialgebraic}
\put(115,18){triviality consistent with}
\put(115,13){a compatible filtration)}
\put(115,3){(Fukui observation)}
\end{picture}
\end{center}

\centerline{Table (*)}

\vspace{7mm}

In the following sections we give a finiteness theorem
on Blow-semialgebraic triviality for a family of the main parts
of 3-dimensional algebraic sets, and our main theorem, namely,
a finiteness theorem on Blow-semialgebraic triviality
consistent with a compatible filtration for a family
of 3-dimensional algebraic sets.
But the situation in the higher dimensional case is more complicated.
In \cite{nakai} I. Nakai constructs a family of polynomial
mappings : $\{ f_a : (\R^3,0) \to (\R^2,0) \}$ in which
topological moduli appear.
On the other hand, T. Fukui \cite{fukui} observes 
that there is a Nash trivial simultaneous resolution for a family 
of 4-dimensional algebraic sets $\{ V_a \}$ in $\R^6$ 
such that the resolution maps represent the Nakai family.
These facts mean that given a simultaneous resolution for a family
of higher dimensional algebraic sets with the parameter algebraic 
set $J$, there does not always exist a thin algebraic subset
$J_1 \subset J$ so that finiteness holds on Blow-semialgebraic
triviality even for the family of the main parts 
of algebraic sets over $J \setminus J_1$.  
Incidentally, the Fukui observation is generalised to a theorem 
of representation by a desingularisation map for any polynomial 
mapping $: \R^n \to \R^p$, $n \ge p$.
For the details of the statement, consult \cite{bekkafukuikoike}.

Using a similar argument to the proof of Theorem III,
we can show the following:

\vspace{3mm}

\noindent {\bf Theorem IV.} {\em Suppose that}
$\dim f_t^{-1}(0) \ge 2$ {\em and}
$\dim f_t^{-1}(0) \cap S(f_t) = 1$ {\em over} $J$.
{\em Then there exists a finite partition} 
$J = Q_1 \cup \cdots \cup Q_u$ {\em such that for each} $i$,

(1) $Q_i$ {\em is a Nash open simplex, and}

(2) $(N \times Q_i,F_{Q_i}^{-1}(0))$ {\em admits a Blow-semialgebraic
trivialisation along} $Q_i$.

\begin{proof}[Proof of Theorem IV] 
By the assumption we see that 
$\dim Sing F_{Q_i}^{-1}(0) \le \dim Q_i + 1$.
Then $\dim \pi_t(S\pi_t) \le 1$ for any $t \in Q_i$.
Therefore the statement follows similarly to Theorem III.
\end{proof}


\bigskip
\section{Finiteness for the main parts of $3$-dimensional algebraic sets.}
\label{finitemainparts}
\medskip

In this section we describe a weak version of our main theorem.
Namely, we show a finiteness theorem on Blow-semialgebraic triviality
for a family of the main parts of $3$-dimensional algebraic sets.

Let $N$ be an affine nonsingular algebraic variety in $\R^m$,
and let $J$ be an algebraic set in $\R^a$.
Let $f_t : N \to \R^k$ ($t \in J$) be a polynomial mapping
such that $\dim f_t^{-1}(0) \le 3$ for $t \in J$.
Assume that $F$ is a polynomial mapping.
Then we have

\begin{prop}\label{mainpart}
There exists a finite partition
$J = Q_1 \cup \cdots \cup Q_u$
which satisfies the following conditions:

(1) Each $Q_i$ is a Nash open simplex,
and $\dim f_t^{-1}(0)$ and $\dim f_t^{-1}(0) \cap S(f_t)$
are constant over $Q_i$.         

(2) For each $i$ where $\dim f_t^{-1}(0) = 3$ and
$\dim f_t^{-1}(0) \cap S(f_t) \ge 1$ over $Q_i$,
there is a Nash simultaneous resolution $\Pi_i : \mathcal{M}_i
\to N \times Q_i$ of $F_{Q_i}^{-1}(0)$ in $N \times Q_i$ such that
$MF_{Q_i}^{-1}(0)$ admits a $\Pi_i$-Blow-semialgebraic trivialisation 
along $Q_i$.

In the case where $\dim f_t^{-1}(0) \le 2$ over $Q_i$
or $\dim f_t^{-1}(0) \cap S(f_t) \le 0$ over $Q_i$,
$(N \times Q_i,F_{Q_i}^{-1}(0))$ admits
a trivialisation listed in table (*).
\end{prop}

Let us start to show the proposition above.
We can derive the following lemma easily from
table (*) in \S 4.

\begin{lem}\label{dimension2}
Let $J_0$ be a semialgebraic subset of $J$ such that
$\dim f_t^{-1}(0) \le 2$ for $t \in J_0$
or $\dim f_t^{-1}(0) \cap S(f_t) \le 0$ for $t \in J_0$.
Then there exists a finite partition
$J_0 = Q_1 \cup \cdots \cup Q_u$
which satisfies the following conditions:

{\em (1)} Each $Q_i$ is a Nash open simplex.

{\em (2)} For each $i$, there is a Nash simultaneous
resolution $\Pi_i : {\mathcal M}_i \to N \times Q_i$ of
$F_{Q_i}^{-1}(0)$ in $N \times Q_i$ over $Q_i$
such that $(N \times Q_i, F_{Q_i}^{-1}(0))$  admits a trivialisation
on $\Pi_i$ listed in table {\em (*)}.
\end{lem}

By Process I in \S 3, there is a subdivision of $J$ into
two semialgebraic sets $J_1$ and $J_2$ such that
\begin{equation}
  \ \begin{cases}
    \ \dim f_t^{-1}(0) \le 2 \ \text{or} \ \dim f_t^{-1}(0) \cap S(f_t) \le 0
     & \text{for $t \in J_1$} \\
    \ \dim f_t^{-1}(0) = 3 \ \text{and} \ \dim f_t^{-1}(0) \cap S(f_t) \ge 1
     & \text{for $t \in J_2.$}
            \end{cases} 
\end{equation}

\noindent Let $\dim J_2 = b$, and let $\Delta$ be the Zariski
closure of $J_2$ in $\R^a$.
Then $J_2 \subset \Delta \subset J$ such that $\dim \Delta = b$.
Since $J - \Delta \subset J_1$, there is a finite partition 
$J - \Delta = Q_1 \cup \cdots \cup Q_u$
which satisfies the conditions in Lemma \ref{dimension2}.
Therefore it suffices to show the case $\Delta = J$.

Let us consider the algebraic set $V = F^{-1}(0)$. Then
$$
V = \{ (x,t) \in N \times \Delta \ | \ \tilde{F}(x,t) = 0 \}
$$
is a ($3 + b$)-dimensional algebraic set in $N \times \R^a$.
We apply the desingularisation theorem of Hironaka (\cite{hironaka1})
to $V \subset N \times \R^a$.
Let $\dim N = n$.
There is an algebraic resolution of $V$ in $N \times \R^a$,
$\Pi : M \to N \times \R^a$,
where $M$ is a smooth algebraic set of dimension ($n + a$).
Namely, $\Pi$ is the composite of a finite sequence of blowings-up
$\sigma_{j+1} : M_{j+1} \to M_j$ with smooth algebraic centres $C_j$
such that:

(1) The critical set of $\Pi$ is a union of algebraic divisors
$D_1, \cdots , D_d$.

(2) The strict transform $V^{\prime}$ of $V$ in $M$ by $\Pi$
is a smooth algebraic set of $M$.

(3) $V^{\prime}, D_1, \cdots , D_d$ simultaneously have only normal
crossings.

(4) There is a thin algebraic subset $T$ in $V$ so that
$\Pi|_{\Pi^{-1}(V-T)} : \Pi^{-1}(V-T) \to V - T$
is a Nash isomorphism.

Let $E_i = V^{\prime} \cap D_i$, $1 \le i \le d$,
and let $E$ be the union of $E_i$'s.
We make some remarks here.

\begin{rem}\label{remark51}
(1) $V^{\prime}$ is a $(3 + b)$-dimensional smooth algebraic set.

(2) Each $E_i$ is a smooth algebraic subset
of $V^{\prime}$ of codimension $1$.

(3) $E_1, \cdots , E_d$ simultaneously have only normal
crossings in $V^{\prime}$.
\end{rem}

\begin{rem}\label{remark52}
Let $U$ be an open semialgebraic subset of $V^{\prime}$.
Let ${\mathcal S}_1$ be a finite Whitney stratification
of $E \cap U$ which is compatible with
$E_{i(1)} \cap \cdots \cap E_{i(\eta )} \cap U$'s
for $1 \le i(1) < \cdots < i(\eta ) \le d$, $1 \le \eta \le d$,
and let ${\mathcal S}_2$ be the set of
connected components of $(V^{\prime} - E) \cap U$.
Then ${\mathcal S} = {\mathcal S}_1 \cup {\mathcal S}_2$
is also a finite Whitney stratification of $U$.
\end{rem}

Let $q : \R^m \times \R^a \to \R^a$ be the
canonical projection.
For $\K = \R$ or $\C$, let $\K P_{[r]}$ denote
the product of $r$ projective lines
$\K P^1 \times \cdots \times \K P^1$.
Since the real projective space is affine and
$\Pi$ is the composite of a finite sequence of
blowings-up with smooth algebraic centres,
we may think $M$ is an affine algebraic set
in some Euclidean space $\R^{\ell}$.
Then we have the following diagram:

\vspace{6mm}

\unitlength 0.1in
\begin{picture}( 27.4000, 13.0000)(  2.0000,-15.0000)
\put(2.0000,-10.4000){\makebox(0,0)[lb]{$\R^m \times \R^a \ \supset \ N \times \R^a \ \supset \ N \times \Delta \ \supset \ V$ }}%
\put(5.3000,-16.7000){\makebox(0,0)[lb]{$\R^a \ \ \ \ = \ \ \ \ \R^a \ \ \ \ \supset \ \ \ \ \Delta$}}%
%
\special{pn 8}%
\special{pa 1416 450}%
\special{pa 1416 810}%
\special{fp}%
\special{sh 1}%
\special{pa 1416 810}%
\special{pa 1436 744}%
\special{pa 1416 758}%
\special{pa 1396 744}%
\special{pa 1416 810}%
\special{fp}%
%
\special{pn 8}%
\special{pa 566 1110}%
\special{pa 566 1470}%
\special{fp}%
\special{sh 1}%
\special{pa 566 1470}%
\special{pa 586 1404}%
\special{pa 566 1418}%
\special{pa 546 1404}%
\special{pa 566 1470}%
\special{fp}%
\put(14.6000,-6.6000){\makebox(0,0)[lb]{$\Pi$}}%
%
\special{pn 8}%
\special{pa 2896 460}%
\special{pa 2896 820}%
\special{fp}%
\special{sh 1}%
\special{pa 2896 820}%
\special{pa 2916 754}%
\special{pa 2896 768}%
\special{pa 2876 754}%
\special{pa 2896 820}%
\special{fp}%
\put(29.4000,-6.5000){\makebox(0,0)[lb]{$\Pi$}}%
\put(6.1500,-13.1000){\makebox(0,0)[lb]{$q$}}%
%
\special{pn 8}%
\special{pa 1426 1100}%
\special{pa 1426 1460}%
\special{fp}%
\special{sh 1}%
\special{pa 1426 1460}%
\special{pa 1446 1394}%
\special{pa 1426 1408}%
\special{pa 1406 1394}%
\special{pa 1426 1460}%
\special{fp}%
\put(14.7500,-13.0000){\makebox(0,0)[lb]{$q$}}%
%
\special{pn 8}%
\special{pa 2276 1100}%
\special{pa 2276 1460}%
\special{fp}%
\special{sh 1}%
\special{pa 2276 1460}%
\special{pa 2296 1394}%
\special{pa 2276 1408}%
\special{pa 2256 1394}%
\special{pa 2276 1460}%
\special{fp}%
\put(23.2500,-13.0000){\makebox(0,0)[lb]{$q$}}%
\put(5.1000,-3.7000){\makebox(0,0)[lb]{$\R^{\ell} \ \ \ \supset \ \ \ \ \ M \ \ \ \ \ \ \ \ \ \ \supset \ \ \ \ \ \ \ \ \ \ V^{\prime} \ \ \ \supset \ \ \ E$}}%
\end{picture}%

\vspace{8mm}

Set $\Pi^E = \Pi |_E$. Then we have
$$
\begin{CD}
\R^{\ell} \supset E @>\text{$\Pi^E$}>> \R^m \times \R^a
@>\text{$q$}>> \R^a.
\end{CD}
$$
We denote by $G$ the graph of $\Pi^E$.
Then $G$ is a $(2 + b)$-dimensional algebraic set in
$\R^{\ell} \times \R^m \times \R^a$.
Through the canonical embeddings, let us consider
$\R^{\ell} \subset \R P^{\ell}$, $\R^m \subset \R P_{[m]}$
and $\R^a \subset \R P_{[a]}$.
Let $W$ be the associated algebraic set of $G$ in
$\R P^{\ell} \times \R P_{[m]} \times \R P_{[a]}$.
We regard $\Pi^E : E \to \R^m \times \R^a$ as
the canonical projection $\pi : G \to \R^m \times \R^a$.
Let $\breve{\Pi}$ be the canonical projection from
$W$ to $\R P_{[m]} \times \R P_{[a]}$.
By construction, $\breve{\Pi} |_{W \cap \R^{\ell} 
\times \R^m \times \R^a} = \pi$.
Let $\breve{q} : \R P_{[m]} \times \R P_{[a]} \to \R P_{[a]}$
be the canonical projection which is an extension of $q$.
Then we have a composite of projections:
\begin{equation}
\begin{CD}
\R P^{\ell} \times \R P_{[m]} \times \R P_{[a]} \supset W
@>\text{$\breve{\Pi}$}>> \R P_{[m]} \times \R P_{[a]}
@>\text{$\breve{q}$}>> \R P_{[a]}.
\end{CD}
\end{equation}
Here we consider the complexification of the above
composite of projections:
\begin{equation}
\begin{CD}
\C P^{\ell} \times \C P_{[m]} \times \C P_{[a]} \supset W_{\C}
@>\text{$\breve{\Pi}_{\C}$}>> \C P_{[m]} \times \C P_{[a]}
@>\text{$\breve{q}_{\C}$}>> \C P_{[a]}.
\end{CD}
\end{equation}

\vspace{3mm}

In this section we want to show finiteness for a family 
of the main parts of 3-dimensional algebraic sets, 
following our programme in \S 3.
We remember that Lemma \ref{dimension2} follows from our programme
with Assumption A in Process IV.
But Assumption A is not always satisfied
in the case where $\dim f_t^{-1}(0) = 3$
for any element $t$ of $J_1$ whose Zariski closure is $J$.
Namely, $\Pi^E$ may contain a family of non-Thom maps,
even after taking a finite subdivision of the parameter space
and considering only on the maximal dimensional subspaces.
Therefore we consider some other method for which
an alternative assumption is satisfied in Process IV.
Here we recall the works of C. Sabbah on ``sans \'eclatement"
stratified analytic morphisms (\cite{sabbah2})
and locally finiteness property on topological equivalence
for a family of complex analytic mappings (\cite{sabbah1}).
We don't describe the original forms of Sabbah's results
but modified ones suitable for our purpose.
We first prepare a terminology for it.
Let $f : A \to B$ be a stratified, complex polynomial mapping
between complex algebraic varieties.
Namely, there are finite stratifications 
$\{ A_{\alpha} \}$ of $A$ and $\{ B_{\beta} \}$ of $B$
such that for any $\alpha$, there exists $\beta$
for which $f$ induces a surjective submersion
$f : A_{\alpha} \to B_{\beta}$.
In this paper, a stratification for an algebraic variety
always means a finite one.
We say that a couple of stratifications 
$\mathcal{S} = (\{ A_{\alpha} \} , \{ B_\beta \} )$
is a $\C$-{\em algebraic stratification of} $f$,
if the following conditions are satisfied:

(1) Each $A_{\alpha}$ (resp. $B_{\beta}$) is a (connected)
complex-analytic manifold of $A$ (resp. $B$).

(2) For each $\alpha$ (resp. $\beta$),
$\overline{A_{\alpha}}$ and $\overline{A_{\alpha}} - A_{\alpha}$
(resp. $\overline{B_{\beta}}$ and $\overline{B_{\beta}} - B_{\beta}$)
are algebraically closed in $A$ (resp. $B$).

(3) For any $A_{\alpha}$, there exists $B_{\beta}$ such that
$f(A_{\alpha}) = B_{\beta}$.

Let $X_{\C}$, $Y_{\C}$ be the products of some complex projective spaces,
and let $\tau : X_{\C} \times Y_{\C} \to Y_{\C}$ 
be the canonical projection.
Let $W_{\C}$ be an algebraic set of $X_{\C} \times Y_{\C}$,
and set $\Upsilon = \tau |_{W_{\C}} : W_{\C} \to Y_{\C}$.

\begin{thm}\label{sabbaht1} (Modified form of \cite{sabbah2}, Theorem 1)
Let $\mathcal{S}$ be a $\C$-algebraic stratification of $\Upsilon$.
Then there exist a proper algebraic modification 
$\varpi : \tilde{Y}_{\C} \to Y_{\C}$
and a $\C$-algebraic stratification $\tilde{\mathcal{S}}$ of
the pull-back of $\Upsilon$ by $\varpi$, $\tilde{\Upsilon} : 
W_{\C} \underset{Y_{\C}}{\times} \tilde{Y}_{\C} \to \tilde{Y}_{\C}$,
compatible with $\mathcal{S}$ such that 
$(\tilde{\Upsilon},\tilde{\mathcal{S}})$ is sans \'eclatement
in the sense of Sabbah (cf. Remark 3.4).
\end{thm}

\begin{rem}\label{remark53}
Let $X$, $Y$ be the products of some real projective spaces,
and let $W$ be a real algebraic set of $X \times Y$.
We denote by $X_{\C}$, $Y_{\C}$ and $W_{\C}$ the complexifications
of $X$, $Y$ and $W$, respectively.
Let $\tau : X_{\C} \times Y_{\C} \to Y_{\C}$ be 
the canonical projection.
Then $\Upsilon = \tau |_{W_{\C}} : W_{\C} \to Y_{\C}$
is invariant under complex conjugation.
Sabbah proved Theorem \ref{sabbaht1} using the flattening theorem
of Hironaka \cite{hironaka1}.
Therefore we can assume that $\varpi$ is invariant under
complex conjugation. 
Suppose that $\mathcal{S}$ is a $\C$-algebraic stratification
of $\Upsilon$ which is invariant under complex conjugation.
Then, thanks to the construction of Sabbah \cite{sabbah2},
we can assume that $(\tilde{\Upsilon},\tilde{\mathcal{S}})$
in Theorem \ref{sabbaht1} is also invariant under complex conjugation.
\end{rem}

Applying Theorem 1 in \cite{sabbah2},
Sabbah established a local finiteness theorem on topological equivalence
for a family of complex analytic mappings defined over a
compact, connected, smooth analytic surface (\cite{sabbah1}, Theorem 2),
and proved a finite classification theorem on topological equivalence 
for complex polynomial mappings : $(\C^2,0) \to (\C^m,0)$
of a bounded degree as its corollary.

Let us recall Sabbah's local finiteness theorem in the
algebraic form, namely a finiteness theorem for a family of 
polynomial mappings as an outcome.
Let $W_{\C}$ and $Y_{\C}$ be compact complex algebraic
varieties, let $U_{\C}$ be a complex algebraic variety,
and let $\mu : W_{\C} \times U_{\C} \to Y_{\C} \times U_{\C}$
be a family of polynomial mappings and $q : Y_{\C} \times U_{\C}
\to U_{\C}$ be the canonical projection.
Set $W_{\C ,u} = (q \circ \mu)^{-1}(u)$ and $Y_{\C ,u} = q^{-1}(u)$
for $u \in U_{\C}$.
Suppose that $W_{\C}$ is nonsingular, connected and of dimension 2.
Then we have
  
\begin{thm}\label{sabbaht2} (\cite{sabbah1} Theorem 2)
The number of topological types of polynomial mappings
$\mu |_{W_{\C ,u}} : W_{\C ,u} \to Y_{\C ,u}$ is finite
over $U_{\C}$.
\end{thm}

Note that Sabbah proved finiteness theorem not only on topological equivalence
but also on topological triviality implicitly in the theorem above.

Before going back to the original stage of our proof
(i.e. the stage at (5.3)), we describe some important
observations on the stratification of a complexified mapping
between complexified spaces which is invariant 
under complex conjugation.

Let $A_{\C} \subset X_{\C}$ be the complexification of 
a real algebraic set $A$ in the product of some real
projective spaces $X$.

\vspace{3mm}

\noindent Observation 1. Let $\mathcal{S}(A_{\C})$
be a $\C$-algebraic stratification of $A_{\C}$
which is invariant under complex conjugation.
We denote by $\mathcal{S}(A)$ the set of connected
components of $P_{\C} \cap X$ for all $P_{\C} \in
\mathcal{S}(A_{\C})$.
Note that some $P_{\C} \cap X$ may be empty.
But $\mathcal{S}(A)$ gives a stratification of $A$
which does not always satisfy the frontier condition.
Each stratum of $\mathcal{S}(A)$ is a Nash manifold
and the real dimension of any connected component
of $P_{\C} \cap X$ (if not empty) is equal
to the complex dimension of $P_{\C}$.

\vspace{3mm}

Since the Whitney $(b)$-regularity is just a property
of the inclusion of limit spaces in the Grassmannian,
we have the following:

\vspace{3mm}

\noindent Observation 2. Let $R, \ U \in \mathcal{S}(A)$ 
with $\overline{R} \cap U \ne \emptyset$, and let $R_{\C}$, 
$U_{\C} \in \mathcal{S}(A_{\C})$ such that $U \subset U_{\C}$
and $R \subset R_{\C}$.
Suppose that $R_{\C}$ is Whitney $(b)$-regular over $U_{\C}$.
Then $R$ is also Whitney $(b)$-regular over $U$ at any point of
$\overline{R} \cap U$.

\vspace{3mm}

Let $A_{\C} \subset X_{\C}$ ($B_{\C} \subset Y_{\C}$)
be the complexification of a real algebraic set 
$A$ (resp. $B$) in $X$ (resp. $Y$) as above, and let 
$f : (A_{\C},\mathcal{S}(A_{\C})) \to (B_{\C},\mathcal{S}(B_{\C}))$
be a stratified polynomial mapping 
which is invariant under complex conjugation such that 
$(\mathcal{S}(A_{\C}),\mathcal{S}(B_{\C}))$
is a $\C$-algebraic stratification of $f$.
We define $\mathcal{S}(B)$ similarly to $\mathcal{S}(A)$ in Observation 1.
Then we have

\vspace{3mm}

\noindent Observation 3. Take any $R \in \mathcal{S}(A)$.
For $P \in R$, let $f(P) \in R^{\prime} \in \mathcal{S}(B)$.
Then, locally around $P$, $f : R \to R^{\prime}$ is a submersion.

\vspace{3mm}

Because of the same reason as the Whitney $(b)$-regularity, we have

\vspace{3mm}

\noindent Observation 4. Let $R, \ U \in \mathcal{S}(A)$ 
with $\overline{R} \cap U \ne \emptyset$, and let $R_{\C}$, 
$U_{\C} \in \mathcal{S}(A_{\C})$ such that $U \subset U_{\C}$
and $R \subset R_{\C}$.
Suppose that $R_{\C}$ is Thom $(a_f)$-regular over $U_{\C}$.
Then $R$ is also Thom $(a_f)$-regular over $U$ at any point of
$\overline{R} \cap U$.

\vspace{3mm}

Now we make some preparations to apply Theorem \ref{sabbaht1} 
with the arguments in \cite{sabbah1} and the above 
observations to maps (5.3) and (5.2).
We first consider the parameter space.
Set $K_{\C} = (\breve{q}_{\C} \circ \breve{\Pi}_{\C})(W_{\C})$,
which is a complex algebraic set in $\C P_{[a]}$.
Let $\dim_{\C} K_{\C} = b_1$.
By construction, we have $b_1 \le b$.
Set $K = (\breve{q}_{\C} \circ \breve{\Pi}_{\C})(W) \
(= (\breve{q} \circ \breve{\Pi})(W))$.
Then we see that $K_{\C} \cap \R P_{[a]}$ is 
an algebraic set in $\R P_{[a]}$, and $K$ is 
a semialgebraic subset of $K_{\C} \cap \R P_{[a]}$.
Note that $\dim K = b_1 \le b$.

We next prepare some notations.
Let $I = \{ 1, 2, \cdots , d \}$.
We identify $E$ with $G$.
Then we have
$$
G = E = \bigcup_{i = 1}^d E_i.
$$
Let $W_i$, $i \in I$, be the
associated algebraic set of $E_i$ in
$\R P^{\ell} \times \R P_{[m]} \times \R P_{[a]}$, and let
$W_{i\C}$,  $i \in I$, be the complexification of $W_i$ 
in $\C P^{\ell} \times \C P_{[m]} \times \C P_{[a]}$.
Then we have

\begin{equation}
W = \bigcup_{i = 1}^d W_i, \ \ \ \
W_{\C} = \bigcup_{i = 1}^d W_{i\C}.
\end{equation}

Let us consider the map 
$\breve{\Pi}_{\C}|_{W_{i\C}}$, $i \in I$.
For simplicity, we set $\Upsilon_i = \breve{\Pi}_{\C}|_{W_{i\C}}$,
$Y_{i\C} = \Upsilon_i (W_{i\C})$ and 
$K_{i\C} = (\breve{q}_{\C} \circ \Upsilon_i )(W_{i\C})$.
Then $Y_{i\C}$ and $K_{i\C}$ are complex algebraic sets 
in $\C P_{[m]} \times \C P_{[a]}$ and $\C P_{[a]}$, respectively.
We stratify the map $\Upsilon_i : W_{i\C} \to Y_{i\C}$ so that
the couple of stratifications 
$\mathcal{S}_i = (\mathcal{S}_{W_{i\C}},\mathcal{S}^0_{Y_{i\C}})$
is a $\C$-algebraic stratification of $\Upsilon_i$
which is invariant under complex conjugation and has 
the following compatibilities:

\vspace{2mm}

(1) $\mathcal{S}_{W_{i\C}}$ is compatible with
$W_{i\C} \cap \C^{\ell} \times \C^m \times \C^a$ and
$W_{i\C} \cap (\bigcap_{\lambda \in \Xi} W_{\lambda\C})$'s
for any subset $\Xi \subset I \setminus \{ i \}$.
(Many of the latter intersections may be empty.)

(2) $\mathcal{S}^0_{Y_{i\C}}$ is compatible with
$Y_{i\C} \cap \C^m \times \C^a$ and the images of all the above subsets
of $W_{i\C}$ by $\Upsilon_i$.

\vspace{2mm}

\noindent Then, using the same argument
as Lemma in \cite{sabbah1}, we have the following:

\begin{lem}\label{lemmasa} 
There exists a proper algebraic modification 
$\varpi_i : \tilde{Y}_{i\C} \to Y_{i\C}$
such that for any algebraic subset $\Sigma_i$ of $Y_{i\C}$,
generically of relative dimension $1$ over $K_{i\C}$,
there are Whitney stratifications $\mathcal{S}_{Y_{i\C}}$ 
and $\mathcal{S}_{\tilde{Y}_{i\C}}$ of $Y_{i\C}$ and 
$\tilde{Y}_{i\C}$, respectively, and
a $\C$-algebraic stratification $\tilde{\mathcal{S}}_i = 
(\mathcal{S}_{\tilde{W}_{i\C}},\mathcal{S}_{\tilde{Y}_{i\C}})$ 
of the pull-back of $\Upsilon_i$ by $\varpi_i$, 
$\tilde{\Upsilon}_i :
\tilde{W}_{i\C} = W_{i\C} \underset{Y_{i\C}}{\times} \tilde{Y}_{i\C}
\to \tilde{Y}_{i\C}$,
which satisfy the following:

{\em (1)} $\mathcal{S}_{\tilde{Y}_{i\C}}$ is compatible with 
$\Gamma_i = \varpi_i^{-1}(\Sigma_i )$, 
and $\mathcal{S}_{Y_{i\C}}$ is a substratification
of $\mathcal{S}^0_{Y_{i\C}}$ compatible with $\Sigma_i$.

{\em (2)} $\varpi_i : (\tilde{Y}_{i\C},\mathcal{S}_{\tilde{Y}_{i\C}})
\to (Y_{i\C},\mathcal{S}_{Y_{i\C}})$ is a stratified mapping.

{\em (3)} $\tilde{\mathcal{S}}_i$ is compatible with $\mathcal{S}_i$, 
and $(\tilde{\Upsilon}_i,\tilde{\mathcal{S}}_i)$ is sans \'eclatement.

{\em (4)} There is a dense, smooth, Zariski open subset $\Omega_i$ of
$K_{i\C}$ such that the restriction over $\Omega_i$ of the
stratified mapping $\varpi_i : \Gamma_i \to \Sigma_i$ 
endowed with the stratifications introduced by
$\mathcal{S}_{\tilde{Y}_{i\C}}$ and $\mathcal{S}_{Y_{i\C}}$ 
is sans \'eclatement.

{\em (5)} If $B_{\beta}$ is a stratum of $\mathcal{S}_{Y_{i\C}}$
such that $\breve{q}_{\C}(\overline{B_{\beta}}) \cap \Omega_i \neq \emptyset$,
then the restriction $\breve{q}_{\C}|_{B_{\beta} \cap \Omega_i} :
B_{\beta} \cap \Omega_i \to \Omega_i$ is a submersion. 
\end{lem}

\begin{rem}\label{remark54}
The image $\Upsilon_i (W_{i\C})$ of $W_{i\C}$ by $\Upsilon_i$
is generically less than or equal to one-dimensional over $K_{i\C}$.
Note that for $j \ne i$, not only $\Upsilon_i (W_{i\C})$ but
$\Upsilon_i (W_{i\C} \cap W_{j\C})$ can be generically one-dimensional over $K_{i\C}$.
In that case, we can take $\Sigma_i$ in Lemma \ref{lemmasa} so that
$\Upsilon_i (W_{i\C} \cap W_{j\C}) \subset \Sigma_i$.
\end{rem}

Let $\hat{\Upsilon}_i : \hat{W}_{i\C} \to \tilde{Y}_{i\C}$ 
be the strict transform of $\Upsilon_i$ by $\varpi_i$.
Since $\hat{W}_{i\C}$ is an irreducible component of $W_{i\C}$,
the above stratification gives a $\C$-algebraic stratification 
$\hat{\mathcal{S}}_i = (\mathcal{S}_{\hat{W}_{i\C}},\mathcal{S}_{\tilde{Y}_{i\C}})$ 
of $\hat{\Upsilon}_i$ so that 
$(\hat{\Upsilon}_i,\hat{\mathcal{S}}_i)$ is sans \'eclatement.
As mentioned in Remark \ref{remark53}, we can assume that 
$\varpi_i$ and $(\hat{\Upsilon}_i,\hat{\mathcal{S}}_i)$ are invariant 
under complex conjugation.
By construction, we can assume also that 
$\mathcal{S}_{Y_{i\C}}$ and $\Omega_i$ are invariant 
under complex conjugation.

Let $h_i$ be the map : 
$\tilde{W}_{i\C} = W_{i\C} \underset{Y_{i\C}}{\times} \tilde{Y}_{i\C}
\to W_{i\C}$ in Lemma \ref{lemmasa} such that
$\Upsilon_i \circ h_i = \varpi_i \circ \tilde{\Upsilon}_i$,
and let $\hat{h}_i = h_i |_{\hat{W}_{i\C}}$.
Then we have the following commutative diagram of
stratified mappings endowed with the aforementioned stratifications:

\begin{equation}\label{complexdiagram}
\minCDarrowwidth 1pt
\begin{CD}
\hat{W}_{i\C} @> \hat{\Upsilon}_i>>  \tilde{Y}_{i\C} \\ 
@V\hat{h}_i VV @VV \varpi_i V\\
W_{i\C} @>\Upsilon_i>> Y_{i\C}  @>\breve{q}_{\C}>> K_{i\C}
\end{CD}
\end{equation}

\vspace{1mm}

Note that the maps $\hat{\Upsilon}_i$ and $\hat{h}_i$ are also
invariant under complex conjugation by construction,
and $\breve{q}_{\C}$ is the restriction of the projection
$\C P_{[m]} \times \C_{[a]} \to \C_{[a]}$
to $Y_{i\C}$.

We next consider the real part of the diagram above.
Let $\hat{W}_i$, $\tilde{Y}_i$ and $W_i$
be the sets of real points of $\hat{W}_{i\C}$, $\tilde{Y}_{i\C}$ and,
$W_{i\C}$, respectively.
Define $\hat{h}_{i\R} : \hat{W}_i \to W_i$ by 
$\hat{h}_{i\R} = \hat{h}_i |_{W_i}$,
$\hat{\Upsilon}_{i\R} : \hat{W}_i \to \tilde{Y}_i$
by $\hat{\Upsilon}_{i\R} = \hat{\Upsilon}_i |_{\hat{W}_i}$,
and $\Upsilon_{i\R} : W_i \to \R P_{[m]} \times \R P_{[a]}$
by $\Upsilon_{i\R} = \Upsilon_i |_{W_i}$.
Recall that
$W_i = W_{i\C} \cap \R P^{\ell} \times \R P_{[m]} \times \R P_{[a]}$
and $\Upsilon_{i\R} = \breve{\Pi} |_{W_i}$.
Set $\hat{Y}_i = \hat{\Upsilon}_{i\R} (\hat{W}_i)$
and $Y_i = \Upsilon_i (W_i)$. 
We further define
$\varpi_{i\R} : \hat{Y}_i \to Y_i$ by 
$\varpi_{i\R} = \varpi_i |_{\hat{Y}_i}$.
Let $K_i = \breve{q}_{\C} \circ \Upsilon_i (W_i)
= \breve{q} \circ \Upsilon_{i\R} (W_i)$.
Then we have the following commutative diagram:

\begin{equation}\label{realdiagram}
\minCDarrowwidth 1pt
\begin{CD}
\hat{W}_i @> \hat{\Upsilon}_{i\R}>>  \hat{Y}_i \\ 
@V\hat{h}_{i\R} VV @VV \varpi_{i\R} V\\
W_i @>\Upsilon_{i\R}>> Y_i  @>\breve{q}>> K_i
\end{CD}
\end{equation}

\vspace{2mm}

Applying the desingularisation theorem of Hironaka,
we can assume from the beginning that $W_i$ and $W_{i\C}$ 
are nonsingular.
Note that the desingularised $W_{i\C}$ and the real part $W_i$
coincide with the original ones in $\C^{\ell} \times \C^m \times \C^a$
and $\R^{\ell} \times \R^m \times \R^a$, respectively.
Let $\Sigma_i$ be the image of the critical points set of $\varpi_i$
in Lemma \ref{lemmasa}.
Using the arguments of the proofs of Theorem \ref{sabbaht2} and
Thom's Isotopy Lemmas, we can see that there are a complex
algebraic subset $H_{i\C}$ of $K_{i\C}$ with 
$\dim_{\C} H_{i\C} < \dim_{\C} K_{i\C}$ and a finite partition
of $K_{i\C} \setminus H_{i\C}$ into Nash open simplices $Q_j$'s
such that the stratified set $\Sigma_i$ and
the stratified mapping $\hat{\Upsilon}_i$ in diagram
(\ref{complexdiagram}) are topologically trivial over each $Q_j$.
By construction, there is a thin algebraic subset $\Theta_i$ of
$W_{i\C}$ such that 
$\hat{h}_i : \hat{W}_{i\C} \setminus \hat{h}_i^{-1}(\Theta )
\to W_{i\C} \setminus \Theta_i$ is an isomorphism.
In addition, since $W_{i\C}$ is nonsingular and 
${W_{i\C}}_{,t} = W_{i\C} \cap \breve{q}_{\C} \circ \Upsilon_i^{-1}(t)$
is complex 2-dimensional for each $t \in K_{i\C}$, we can assume that 
$\Theta_{i,t} = \breve{q}_{\C} \circ \Upsilon_i^{-1}(t) \cap \Theta_i$ 
is a finitely many points for any $t \in Q_j$ and $j$.
Then it follows from the argument of the proof of Theorem \ref{sabbaht2}
that the topological triviality of $\hat{\Upsilon}_i$ over $Q_j$
induces a topological one of $\Upsilon_i$ over $Q_j$.
By the proof of the theorem, we can see that the topological triviality 
of $\tilde{Y}_{i\C}$ in $\hat{\Upsilon}_i$ over $Q_j$ is an extension 
of the lifting of the topological triviality of $\Sigma_i$ over $Q_j$.
Therefore let us remark that the induced topological triviality
of $Y_{i\C}$ in $\Upsilon_i$ over $Q_j$ is an extension
of the topological triviality of $\Sigma_i$ over $Q_j$ for each $j$.
We make one more important remark that $\dim H_{i\C} \cap \R P_{[a]}
< \dim_{\C} K_{i\C}$.

Let us consider the real diagram (\ref{realdiagram}),
keeping the above observation in the complex case.
By definition we have $K = \bigcup_{i=1}^d K_i$.
Since $\Delta = J$,
there exists a finite partition $J = Q_1 \cup \cdots \cup Q_u \cup R$
which satisfies the following conditions:

(1) Each $Q_j$ is an open Nash simplex, and $R$ is a semialgebraic subset
of $J$ with $\dim R < \dim J$.

(2) For any $t \in J \setminus R$, $\dim F_t^{-1}(0) = 3$.

Assume that $\dim K < b$.
Then we take $R$ so that $K \subset R$.
Using the arguments of the proofs of Proposition \ref{reduction}
and some other finiteness theorem (e.g. Theorem II), we can assume
that each $F_{Q_j}^{-1}(0)$ is Nash trivial over $Q_j$.

In the following we assume that $\dim K = b$, namely, there is $K_i$ 
such that $\dim K_i = b$.
Let $\Sigma_{i\R} = \Sigma_i \cap \R P_{[m]} \times \R P_{[a]}$ and
$\Theta^{\prime}_i = \Theta_i \cap 
\R P^{\ell} \times \R P_{[m]} \times \R P_{[a]}$.
Then, by the construction of stratified sets and mappings
in diagram (\ref{complexdiagram}) and Observations 1-4, 
we can see that there exist a semialgebraic subset $R_i$ of $K_i$ with
$\dim R_i < b$ and a finite partition $K_i \setminus R_i = 
Q_{i,1} \cup \cdots \cup Q_{i,u}$ into Nash open simplices with 
$\dim Q_{i,j} = b$ so that over each $Q_{i,j}$,
there are a couple of finite $C^{\omega}$ Nash Whitney stratifications 
$\mathcal{S}_{i\R} = (\mathcal{S}_{W_i},\mathcal{S}_{Y_i})$ of $(W_i,Y_i)$, 
and finite $C^{\omega}$ Nash Whitney stratifications 
$\mathcal{S}_{\hat{Y}_i}$ of $\hat{Y}_i$ and $\mathcal{S}_{\hat{W}_i}$
of $\hat{W}_i$ respectively, which satisfy the following conditions:

(1) The map $\Upsilon_{i\R} : W_i \to Y_i$ with $\mathcal{S}_{i\R}$
is a stratified mapping.

(2) $\mathcal{S}_{W_i}$ is compatible with
$W_i \cap \R^{\ell} \times \R^m \times \R^a$ and
$W_i \cap (\bigcap_{\lambda \in \Xi} W_{\lambda})$'s
for any subset $\Xi \subset I \setminus \{ i \}$.

(3) $\mathcal{S}_{Y_i}$ is compatible with
$Y_{i\R} \cap \R^m \times \R^a$ and the images of all the above subsets
of $W_i$ by $\Upsilon_{i\R}$.

(4) The map $\hat{\Upsilon}_{i\R} :
\hat{W}_i \to \hat{Y}_i$ with $\hat{\mathcal{S}}_{i\R} = 
(\mathcal{S}_{\hat{W}_i},\mathcal{S}_{\hat{Y}_i})$ is a stratified mapping,
and $(\hat{\Upsilon}_{i\R},\hat{\mathcal{S}}_{i\R})$ is sans \'eclatement.
$\hat{\mathcal{S}}_{i\R}$ is compatible with $\mathcal{S}_{i\R}$.

(5) $\mathcal{S}_{\hat{Y}_i}$ is compatible with 
$\Gamma_{i\R} = \varpi_{i\R}^{-1}(\Sigma_{i\R} )$, 
and $\mathcal{S}_{Y_i}$ is compatible with $\Sigma_{i\R}$.

(6) $\varpi_{i\R} : (\hat{Y}_i,\mathcal{S}_{\hat{Y}_i})
\to (Y_i,\mathcal{S}_{Y_i})$ is a stratified mapping.

(7) The restricted stratified mapping 
$\varpi_{i\R} : \Gamma_{i\R} \to \Sigma_{i\R}$ 
endowed with the stratifications introduced by
$\mathcal{S}_{\hat{Y}_i}$ and $\mathcal{S}_{Y_i}$ is sans \'eclatement.

(8) $\breve{q} : \breve{q}^{-1}(Q_{i,j}) \cap Y_i \to Q_{i,j}$
is proper.
If $B_{\beta}$ is a stratum of $\mathcal{S}_{Y_i}$,
then the restriction $\breve{q}|_{B_{\beta}} :
B_{\beta} \to Q_{i,j}$ is a submersion.

\vspace{2mm}

\noindent Using a similar argument to the above with the semialgebraic 
versions of Thom's Isotopy Lemmas (cf. Lemmas \ref{2ndisotopy},
\ref{1stisotopy}) and Remark \ref{remark32}, 
we can see that the stratified set $\Sigma_{i\R}$ and
the stratified mapping $\hat{\Upsilon}_{i\R}$ in diagram
(\ref{realdiagram}) are semialgebraically trivial over each $Q_{i,j}$.
In addition, taking a finite subdivision of $Q_{i,j}$'s 
and taking $R_i$ bigger if necessary, we may assume that
$\Theta_i^{\prime}$ is a thin semialgebraic subset of 
$W_i$ over $Q_{i,j}$ such that for each $j$, 
$\hat{h}_{i\R} : \hat{W}_i \setminus \hat{h}_{i\R}^{-1}(\Theta_i^{\prime} )
\to W_i \setminus \Theta_i^{\prime}$ 
is a Nash isomorphism and
$\Theta_i^{\prime}$ is Nash trivial over $Q_{i,j}$, and 
$\Theta^{\prime}_{i,t} = \breve{q} \circ \Upsilon_{i\R}^{-1}(t) 
\cap \Theta_i^{\prime}$ is a finitely many points for any $t \in Q_{i,j}$.
Then it follows from a similar argument to the above that the 
semialgebraic triviality of $\hat{\Upsilon}_{i\R}$ over $Q_{i,j}$
induces a semialgebraic one of $\Upsilon_{i\R}$ over $Q_{i,j}$,
and the induced semialgebraic triviality of $Y_i$ in $\Upsilon_{i\R}$ 
over $Q_{i,j}$ is an extension of the semialgebraic triviality of 
$\Sigma_{i\R}$ over $Q_{i,j}$ for each $j$.

We show that after removing a thin semialgebraic subset $R$ of $K$
from $K$ and taking a finite subdivision of $K \setminus R$ into Nash open
simplices $Q_k$'s, $\breve{\Pi} : W \to \breve{\Pi}(W) \subset
\R P_{[m]} \times \R P_{[a]}$
is semialgebraically trivial over each $Q_k$.
By construction, $\Upsilon_{i\R} = \breve{\Pi} |_{W_i}$, $i \in I$.
For simplicity let us assume that $\dim K_i = b$, $1 \le i \le u$,
and $\dim K_i < b$, $u + 1 \le i \le d$.
As seen above, there exist a semialgebraic subset $R$ of $K$ 
with $\dim R < b$ and a finite partition of $K_i \setminus R$, 
$1 \le i \le u$, into Nash open simplices $Q_{i,j}$'s, $1 \le j \le s(i)$,
with $\dim Q_{i,j} = b$ for each $i$, $j$ such that
\begin{equation}\label{partition}
K \setminus R = Q_{1,1} \cup \cdots \cup Q_{1,s(1)} \cup \cdots \cup
Q_{u,1} \cup \cdots \cup Q_{u,s(u)} \ \ (u \le d),
\end{equation}
and each $\Upsilon_{i\R}$, $1 \le i \le u$, is semialgebraically trivial 
over $Q_{i,j}$ for $1 \le j \le s(i)$.
If $W_i \cap W_k \ne \emptyset$ for $1 \le i \ne k \le u$ over
$Q_{i,j(i)} \cap Q_{k,j(k)}$,
the semialgebraic trivialities of $\Upsilon_{i\R}$ and $\Upsilon_{k\R}$
may not coincide over $Q_{i,j(i)} \cap Q_{k,j(k)}$.
Therefore we have to modify those semialgebraic trivialities.
Taking a finite subdivision of $Q_{i,j}$'s 
and taking $R$ bigger if necessary, we may assume that
the stratifications of $\Upsilon_{i\R}$'s are compatible
with $W_i \cap W_j$, $\Upsilon_{i\R}(W_i \cap W_k)$, 
$\Upsilon_{k\R}(W_i \cap W_k)$ ($1 \le i <k \le u$) and 
$W_i \cap W_k \cap W_v$, $\Upsilon_i(W_i \cap W_k \cap W_v)$,
$\Upsilon_k(W_i \cap W_k \cap W_v)$,
$\Upsilon_v(W_i \cap W_k \cap W_v)$  ($1 \le i < k < v \le u$).
Many of the above intersections and their images may be empty.
But if not empty, we may assume also the following:

(1) $Q_{i,j(i)} = Q_{k.j(k)} = Q_{i,j(i)} \cap Q_{k,j(k)}$.
Over $Q_{i,j(i)}$, $W_i \cap W_k$ is $1$-dimensional, the image
by $\Upsilon_{i\R}$ (or $\Upsilon_{k\R}$) is $\le 1 $-dimensional, 
and they are semialgebraically trivial.

(2) $Q_{i,j(i)} = Q_{k,j(k)} = Q_{v,j(v)}
= Q_{i,j(i)} \cap Q_{k,j(k)} \cap Q_{v,j(v)}$. 
Over $Q_{i,j(i)}$, $W_i \cap W_k \cap W_v$ and the image by $\Upsilon_{i\R}$
(or $\Upsilon_{k\R}$, $\Upsilon_{v\R}$) is $0$-dimensional, 
and they are semialgebraically trivial.

\noindent Here we make one remark.

\begin{rem}\label{remark55} (1) Since $W^{i,k,v} = W_i \cap W_k \cap W_v$
is $0$-dimensional and semialgebraically trivial over $Q_{i,j(i)}$, 
the semialgebraic trivialities of  $\Upsilon_{i\R}$, $\Upsilon_{k\R}$
and $\Upsilon_{v\R}$ over $Q_{i,j(i)}$ are uniquely determined
on $W^{i,k,v}$.

(2) Let $W^{i,k} = W_i \cap W_k$.
We take the Zariski closure $\Lambda^{i,k}$ of $\Upsilon_{i\R}(W^{i,k})$ 
($= \Upsilon_{k\R}(W^{i,k}))$ in $\R P_{[m]} \times \R P_{[a]}$, 
and consider the intersection of $\breve{q}_{\C}^{-1}(K_{i\C})$ 
and the complexification $\Lambda^{i,k}_{\C}$ of $\Lambda^{i,k}$ 
in $\C P_{[m]} \times \C P_{[a]}$, and also the intersection of 
$\breve{q}_{\C}^{-1}(K_{k\C})$ and $\Lambda^{i,k}_{\C}$.
By Remark \ref{remark54}, we can take the intersections
in $\Sigma_i$ and $\Sigma_k$ in Lemma \ref{lemmasa}.
Therefore, taking a finite subdivision of $Q_{i,j(i)}$
and removing a thin semialgebraic subset if necessary,
we can assume that $\Upsilon_{i\R}(W^{i,k}) = \Upsilon_{k\R}(W^{i,k})$
is in $\Sigma_{i\R}$ and $\Sigma_{k\R}$ over $Q_{i,j(i)}$.

As seen above, given a semialgebraic triviality of $\Sigma_{i\R}$ 
over $Q_{i,j(i)}$, there is a semialgebraic triviality of 
$\hat{\Upsilon}_{i\R}$ over $Q_{i,j(i)}$ such that the triviality 
induces a semialgebraic one of $\Upsilon_{i\R}$ over $Q_{i,j(i)}$, 
and the semialgebraic triviality of $\Gamma_{i\R} \subset \hat{Y}_i$ 
in $\hat{\Upsilon}_{i\R}$ over $Q_{i,j(i)}$ is the lifting of the given 
triviality of $\Sigma_{i\R}$.
Note that this property holds for any $i$, $j(i)$.
The semialgebraic triviality of $\Upsilon_{i\R}$ over $Q_{i,j(i)}$
gives a semialgebraic one of $\Upsilon_{k\R} = \Upsilon_{i\R} : W^{i,k} 
\to \Upsilon_{k\R}(W^{i,k}) = \Upsilon_{i\R}(W^{i,k})$ over 
$Q_{k,k(i)} = Q_{i,j(i)}$.
Since $W^{i,k}$ is $\le 1$-dimensional over $Q_{k,j(k)}$,
taking a finite subdivisions of $Q_{k,j(k)}$
and removing a thin semialgebraic subset if necessary,
we can assume that $\Upsilon_{k\R} : W^{i,k} \to \Upsilon_{k\R}(W^{i,k})$
and $h_{k\R} : h_{k\R}^{-1}(W^{i,k}) \to W^{i.k}$ are Thom maps.
Therefore, by the above property and this Thom $(a_f)$-regularity,
we can construct the semialgebraic triviality of 
$\hat{\Upsilon}_{k\R}$ as an extension of the lifting of the semialgebraic
triviality of $\Upsilon_{k\R} : W^{i,k} \to \Upsilon_{k\R}(W^{i,k})$
over $Q_{k,j(k)}$.
Thus we can modify the semialgebraic triviality of $\Upsilon_{k\R}$
so that the restriction of the semialgebraic triviality to $W^{i,k}$
coincides with the restricted semialgebraic one of $\Upsilon_{i\R}$
over $Q_{k,j(k)}$.
\end{rem}

\noindent Now we can assume the following properties in (\ref{partition}):

(1) If $W_{i,j(i)} \cap W_{k,j(k)} \ne \emptyset$,
then $Q_{i,j(i)} = Q_{k,j(k)}$.
If $W_{i,j(i)} \cap W_{k,j(k)} = \emptyset$,
then $Q_{i,j(i)} \cap Q_{k,j(k)} = \emptyset$.
Here $W_{i,j(i)} = (\breve{q} \circ \Upsilon_{i\R})^{-1}(Q_{i.j(i)})$.

(2) $\Upsilon_{i\R}$ is semialgebraically trivial over $Q_{i,j(i)}$
for any $i$ and $j(i)$.

\noindent We first consider the semialgebraic triviality
of $\Upsilon_{1\R}$ over any $Q_{1,j(1)}$,
Let $k > 1$.
In the case where $Q_{i,j(i)} \cap Q_{k,j(k)} = \emptyset$ for any
$i < k$ and $j(i)$, we consider the semialgebraic triviality
of $\Upsilon_{k\R}$ over any $Q_{k,j(k)}$.
In the case where $Q_{i,j(i)} = Q_{k,j(k)}$ for $i < k$, by Remark 
\ref{remark55}, we can take the the semialgebraic triviality of 
$\Upsilon_{k\R}$ over $Q_{k,j(k)}$ so that the restriction of the 
semialgebraic triviality to $W^{i,k}$ coincides with the restricted 
semialgebraic one of $\Upsilon_{i\R}$ over $Q_{k,j(k)}$.
By this construction, we can see that $\breve{\Pi} : W \to \breve{\Pi}(W)$ 
is semialgebraically trivial over each $Q_{i,j}$.

In order to prove the proposition, that is, finiteness on
Blow-semialgebraic triviality for the main part $M(V)$,
we need to show finiteness on semialgebraic triviality of
$\Pi |_{V^{\prime}} : V^{\prime} \to MV$ after removing a thin semialgebraic 
subset from the original parameter algebraic set $J$.
As stated above, if $\dim K < b$, then finiteness holds on Nash triviality 
for $F_{J \setminus K}^{-1}(0)$.
Therefore it suffices to consider the case that $\dim K = b$.
In this case also, finiteness holds on Nash triviality 
for $F_{J \setminus K}^{-1}(0)$.
Let us restrict our finiteness problem over $K$.
As above we identify $V^{\prime}$ with the graph
of $\Pi |_{V^{\prime}}$, and keep the same notations $W_i$ and $W_{i\C}$.
Let us denote by $\mathcal{V}^{\prime}$ and $\mathcal{D}_j$, 
$1 \le j \le d$, the associated algebraic sets
of $V^{\prime}$ and $D_j$ in 
$\R P^{\ell} \times \R P_{[m]} \times \R P_{[a]}$, respectively.
Similarly, let $\mathcal{D}_{j\C}$, $1 \le j \le d$, be
the associated algebraic set of $D_j$ in 
$\C P^{\ell} \times \C P_{[m]} \times \C P_{[a]}$.
Let us apply our programme to $\breve{\Pi} |_{\mathcal{V}^{\prime}} 
: \mathcal{V}^{\prime} \to \breve{\Pi}(\mathcal{V}^{\prime})$.
Similarly to the above $W_i$ and $W_{i\C}$, applying the desingularisation
theorem of Hironaka, we can assume from the beginning that 
$\mathcal{V}^{\prime}$, $\mathcal{D}_j$ and $\mathcal{D}_{j\C}$, 
$1 \le j \le d$, are nonsingular.
(Some new exceptional divisors may appear by the desingularisation, 
but we use the same number $d$ for simplicity.)
As seen above, finiteness holds on semialgebraic triviality for
$\breve{\Pi} |_W : W \to \breve{\Pi}(W)$ after removing 
a semialgebraic subset $R$ from $K$ with $\dim R < b$.
Namely, there is a finite subdivision of $K \setminus R$ into Nash open
simplices $Q_i$'s such that $\breve{\Pi} |_W : W \to \breve{\Pi}(W)$
is semialgebraically trivial over each $Q_i$.
Since Assumption B is satisfied for each 
$\Pi_i |_{\mathcal{V}^{\prime} \cap (\mathcal{D}_1 \cup \cdots \cup 
\mathcal{D}_d)}$ in Process IV, finiteness holds also on semialgebraic 
triviality for $\breve{\Pi} |_{\mathcal{V}^{\prime}} 
: \mathcal{V}^{\prime} \to \breve{\Pi}(\mathcal{V}^{\prime})$
after removing a thin semialgebraic subset $R$ from $K$.
On the other hand, the stratifications of $W$ and $\breve{\Pi}(W)$ 
over each $Q_i$ are, by construction, compatible with 
$\R^{\ell} \times \R^m \times Q_i$ and $\R^m \times Q_i$, respectively.
Thanks to the compatibility of the above stratifications, finiteness holds 
also on semialgebraic triviality for $\Pi |_{V^{\prime}} : V^{\prime} \to MV$
after removing a thin semialgebraic subset $R$ from $K$.
Namely, finiteness holds on Blow-semialgebraic triviality
for $MV$ over $K \setminus R$.
Thus we have shown the following:

\vspace{3mm}

\noindent Assertion. There exists a finite partition
$J = Q_1 \cup \cdots \cup Q_u \cup R$
such that $R$ is a semialgebraic subset of $J$ with
$\dim R < \dim J$, and statements (1) and (2) in the proposition
hold for each $Q_i$.

\vspace{3mm}

If $\dim f_t^{-1}(0) \le 2$ or $\dim f_t^{-1}(0) \cap S(f_t) \le 0$
for any $t \in R$, then the proposition follows immediately 
from Lemma \ref{dimension2}.
But in the case where $\dim f_t^{-1}(0) = 3$ and
$\dim f_t^{-1}(0) \cap S(f_t) \ge 1$ for some $t \in R$,
we cannot apply the arguments above directly
to the family $\{ f_t : N \to \R^k \}_{t \in R}$,
since $R$ may not be an algebraic set in $\R^a$.
Then we take the Zariski closure $\overline{\overline{R}}$
of $R$ in $\R^a$.
Taking a finite subdivision of each
$Q_i \setminus \overline{\overline{R}}$ into Nash open
simplices $Q_{i_j}$'s if necessary, valid is the statement of the 
assertion replaced $Q_i$'s and $R$ with $Q_{i_j}$'s
and $\overline{\overline{R}}$, respectively.
Therefore in the assertion above, we may assume that
$R$ is a thin algebraic subset of $J$.
Since $\dim R < \dim J$, we can show the proposition
by induction on the dimension of the parameter algebraic set.

This completes the proof of the proposition.

\begin{rem}\label{keyremark}
In the proof of Proposition \ref{mainpart},
we have a partition of $J$
$$
J = (J \setminus K) \cup Q_1 \cup \cdots \cup Q_u \cup R
$$
which satisfies the following conditions:

(1) Finiteness holds on Nash triviality for $F_{J \setminus K}^{-1}(0)$.

(2) For each $j$, $MV_j = MF_{Q_j}^{-1}(0)$ 
admits a $\Pi_j$-Blow-semialgebraic trivialisation along $Q_j$.
Here $M_j = (q \circ \Pi)^{-1}(Q_j)$, and $\Pi_j$ is the restriction
of $\Pi$ to $M_j$.

(3) $\dim Q_j = \dim K$, $1 \le j \le u$, and $\dim R < \dim K$.

Given a semialgebraic subset $L$ of $MV$ which is $\le 1$-dimensional
over $J$.
Let $L_j = L \cap q^{-1}(Q_j)$ for $1 \le j \le u$.
Taking finite subdivisions of $Q_j$'s and removing a bigger $R$ if necessary,
we may assume the following:

(1) $L_j$ is $1$-dimensional and semialgebraically trivial
over $Q_j$ for $1 \le j \le s$.

(2) $L_j$ is $0$-dimensional and semialgebraically trivial
over $Q_j$ for $s + 1 \le j \le v$.

(3) $L_j$ is empty over $Q_j$ for $v + 1 \le j \le u$.

Consider the Zariski closure $\overline{\overline{L}}$ of $L$ in 
$\R^m \times \R^a$, the associated algebraic set $S$ of 
$\overline{\overline{L}}$ in $\R P_{[m]} \times \R P_{[a]}$ 
and the complexification $S_{\C} \subset \C P_{[m]} \times \C P_{[a]}$
of $S$.
Then each $S_{\C} \cap Y_{i\C}$, $1 \le i \le d$, is generically of 
relative dimension $\le 1$ over $K_{i\C}$.
Therefore, using a similar argument to the proof of the proposition above,
we can show that each $L_j$ is semialgebraically trivialised
by the induced semialgebraic triviality of $MV_j$ from the
$\Pi_j$-Blow-semialgebraic triviality,
after taking finite subdivisions of $Q_j$'s and removing 
a bigger $R$ if necessary.

Next let $L_j$ be a semialgebraic subset of $MV_j$ appearing dependently
on $MV_j$ which is semialgebraically trivial and $0$-dimensional or
$1$-dimensional over $Q_j$.
Considering the Zariski closure of $L_j$, the associated algebraic set
and its complexification also in this case,
we can show that $L_j$ is semialgebraically trivialised
by the induced semialgebraic triviality of $MV_j$ from the
$\Pi_j$-Blow-semialgebraic triviality,
after taking a finite subdivision of $Q_j$ and removing also
a thin semialgebraic subset $R_j$ of $Q_j$ if necessary.
Since we take the Zariski closure of $L_j$,
it may affect another $MV_k$, $k \ne j$.
But, when we consider $\Pi_k$-Blow-semialgebraic trivialities
for $k \ne j$, we need not take the affection into consideration.

The latter fact takes a very important role in
the proof of our main theorem given in the next section.
\end{rem}


\bigskip
\section{Proof of the main theorem.}
\label{proof}
\medskip

Given an algebraic set $V \subset \R^m$, we denote by $V^{(1)}$
the Zariski closure of $V \setminus MV$.
Then it is easy to see the following properties on $V^{(1)}$:

\vspace{3mm}

\noindent Property 1. $V \setminus MV = V^{(1)} \setminus MV$.

\vspace{3mm}

\noindent Property 2. $\dim (V^{(1)} \setminus MV) = \dim V^{(1)}
\ge \dim (V^{(1)} \cap MV)$.

\vspace{3mm}

Let $N$ be a nonsingular algebraic variety,
and let $Q \subset \R^m$ be a (connected) Nash manifold.
Let $F : N \times Q \to \R^k$ be a polynomial mapping,
and let $q : N \times Q \to Q$ be the canonical projection.
Set $V = F^{-1}(0)$.
Let $Q = Q_1 \cup \cdots \cup Q_w$ be a finite partition of $Q$
into Nash open simplices, and let
$V_j = V \cap q^{-1}(Q_j)$ and $V_j^{(1)} = V^{(1)} \cap q^{-1}(Q_j)$ 
for $1 \le j \le w$.
We assume that $V$ is semialgebraically trivial over $Q$.
Under this assumption, we have

\begin{lem}\label{samedim}
Let $W_j^{(1)}$ be the Zariski closure of $V_j \setminus MV_j$
in $N \times Q_j$ for $1 \le j \le w$.
Then we have 

\vspace{3mm}

\centerline{$V_j^{(1)} \supset W_j^{(1)}$ and $W_j^{(1)} \setminus MV_j
= V_j^{(1)} \setminus MV_j$ for $1 \le j \le w$.}

\vspace{3mm}

\noindent In addition, if $\dim Q_j = \dim Q$,
then $\dim V_j^{(1)} = \dim W_j^{(1)}$.
It follows that if $\dim Q_j = \dim Q$, then $\dim V_j > \dim V_j^{(1)}$.
\end{lem}

\begin{proof}
By definition, it is obvious that $V_j^{(1)} \supset W_j^{(1)}$
for $1 \le j \le w$.
By Property 1, 
$$
W_j^{(1)} \setminus MV_j = V_j \setminus MV_j \supset
V_j^{(1)} \setminus MV_j \supset W_j^{(1)} \setminus MV_j.
$$
It follows that $W_j^{(1)} \setminus MV_j = V_j^{(1)} \setminus MV_j$ 
for $1 \le j \le w$.

By assumption, $V \setminus MV$ is semialgebraically trivial over $Q$.
Let $V_{[t]} = V \cap q^{-1}(t)$ for $t \in Q$.
Then $\dim (V_{[t]} \setminus MV_{[t]})$ is constant over $Q$.
Put $\dim Q =  a$ and $\dim (V_{[t]} \setminus MV_{[t]}) = b$.
Since $V^{(1)}$ is the Zariski closure of $V \setminus MV$, we have
$$
\dim V^{(1)} = \dim (V \setminus MV) = a + b.
$$
Suppose that $\dim Q_j = \dim Q$. Then we have
$$
\dim W_j^{(1)} = \dim (V_j \setminus MV_j) = a + b. 
$$
On the other hand, $\dim (V_j \setminus MV_j) \le \dim V_j^{(1)}
\le \dim V^{(1)}$.
Thus $\dim V_j^{(1)} = a + b$.
It follows that $\dim V_j^{(1)} = \dim W_j^{(1)}$.
\end{proof}

Let us recall our main theorem.
Let $N$ be an affine nonsingular algebraic variety in $\R^m$,
and let $J$ be an algebraic set in $\R^a$.
Let $f_t : N \to \R^k$ ($t \in J$) be a polynomial mapping
such that $\dim f_t^{-1}(0) \le 3$ for $t \in J$.
Assume that $F$ is a polynomial mapping.
Then we have the following:

\vspace{3mm}

\noindent {\bf Main Theorem.} 
{\em There exists a finite partition $J = Q_1 \cup \cdots \cup Q_u$
which satisfies the following conditions}:

(1) {\em Each $Q_i$ is a Nash open simplex,
and $\dim f_t^{-1}(0)$ and $\dim f_t^{-1}(0) \cap S(f_t)$
are constant over $Q_i$.}

(2) {\em For each $i$ where $\dim f_t^{-1}(0) = 3$ and
$\dim f_t^{-1}(0) \cap S(f_t) \ge 1$ over $Q_i$, 
$F_{Q_i}^{-1}(0)$ admits a Blow-semialgebraic trivialisation 
consistent with a compatible filtration along $Q_i$.

In the case where $\dim f_t^{-1}(0) \le 2$ over $Q_i$
or $\dim f_t^{-1}(0) \cap S(f_t) \le 0$ over $Q_i$,
$(N \times Q_i,F_{Q_i}^{-1}(0))$ admits
a trivialisation listed in table (*).}

\begin{proof} Let us show our main theorem using the results
given in table (*) with Proposition \ref{mainpart}
instead of Main result.
Note that the finiteness theorem in Proposition \ref{mainpart}
holds in the algebraic parameter case, but the others in table (*)
hold in the semialgebraic parameter case.
We sometimes use the above results without referring in this proof.

By Proposition \ref{mainpart}, there is a finite partition
$J = Q_1 \cup \cdots \cup Q_r \cup Q_{r+1} \cup \cdots \cap Q_v$
which satisfies the following:

(1) Each $Q_i$ is a Nash open simplex, and $\dim f_t^{-1}(0)$
and $\dim f_t^{-1}(0) \cap S(f_t)$ are constant over $Q_i$.
In addition, each $F_{Q_i}^{-1}(0)$ is semialgebraically
trivial over $Q_i$.

(2) For $1 \le i \le r$, $\dim f_t^{-1}(0) = 3$ and
$\dim f_t^{-1}(0) \cap S(f_t) \ge 1$ over $Q_i$.
In this case, there is a Nash simultaneous resolution 
$\Pi_i : \mathcal{M}_i \to N \times Q_i$ of 
$F_{Q_i}^{-1}(0)$ in $N \times Q_i$ such that $MF_{Q_i}^{-1}(0)$ 
admits a $\Pi_i$-Blow-semialgebraic trivialisation along $Q_i$.

(3) For $r+1 \le i \le v$, $\dim f_t^{-1}(0) \le 2$ over $Q_i$
or $\dim f_t^{-1}(0) \cap S(f_t) \le 0$ over $Q_i$.
In this case, $(N \times Q_i,F_{Q_i}^{-1}(0))$ admits a trivialisation 
listed in table (*).

\noindent Therefore it suffices to show the first half part of
statement (2) in the main theorem.

Let $V_i = F_{Q_i}^{-1}(0)$, and let $V_i^{(1)}$ be the Zariski
closure of $V_i \setminus MV_i$ in $N \times Q_i$ for $1 \le i \le r$.
By Property 1, each $V_i^{(1)} \setminus  MV_i$ is semialgebraically 
trivial over $Q_i$.
Let $q : N \times J \to J$ be the canonical projection as above.

We first consider the case where $V_i^{(1)} \cap MV_i = \emptyset$.
Then there are 4 possibilities for the dimension of 
$V_i^{(1)} \setminus MV_i = V_i^{(1)}$.
Therefore we subdivide the first case into the following 4 cases:

\vspace{2mm}

(I;-1) \ $V_i = MV_i$ (then $V_i^{(1)} = \emptyset$).

(I;0) \ $\dim (V_i^{(1)} \setminus MV_i) = \dim Q_i$.

(I;1) \ $\dim (V_i^{(1)} \setminus MV_i) = \dim Q_i + 1$.

(I;2) \ $\dim (V_i^{(1)} \setminus MV_i) = \dim Q_i + 2$.

\vspace{2mm}

\noindent {\em Case} (I;-1): $V_i$ already admits a Blow-semialgebraic
trivialisation consistent with the trivial filtration
$\{ V_i \}$ along $Q_i$.

\vspace{2mm}

\noindent {\em Case} (I;0): There is a finite partition
$Q_i = Q_{i,1} \cup \cdots \cup Q_{i,w(i)}$ which satisfies the following:

(I;0:1) Each $Q_{i,j}$ is a Nash open simplex.

(I;0:2) For each $j$, $V_i^{(1)} \cap q^{-1}(Q_{i,j})$ is Nash diffeomorphic
to the direct product of some finite points in $V_i^{(1)} \cap q^{-1}(t)$,
$t \in Q_{i,j}$, and $Q_{i,j}$.

Let $V_{i,j} = V_{i,j}^{(0)} = V_i \cap q^{-1}(Q_{i,j})$ and
$V_{i,j}^{(1)} = V_i^{(1)} \cap q^{-1}(Q_{i,j})$ for $1 \le j \le w(i)$.
Then each $V_{i,j}$ admits a Blow-semialgebraic trivialisation
consistent with the canonical filtration
$\{ V_{i,j}^{(0)} \supset V_{i,j}^{(1)} \}$ along $Q_{i,j}$.

\vspace{2mm}

\noindent {\em Case} (I;1): There is a finite partition
$Q_i = Q_{i,1} \cup \cdots \cup Q_{i,w(i)}$ which satisfies the following:

(I;1:1) Each $Q_{i,j}$ is a Nash open simplex.

(I;1:2) For each $j$, there is a Nash simultaneous resolution 
$\beta_{i,j} : \mathcal{M}_{i,j} \to N \times Q_{i,j}$ of 
$V_i^{(1)} \cap q^{-1}(Q_{i,j})$ in $N \times Q_{i,j}$ such that
$(N \times Q_{i,j},V_i^{(1)} \cap q^{-1}(Q_{i,j}))$ admits 
a $\beta_{i,j}$-Blow-Nash trivialisation along $Q_{i,j}$.

Let $V_{i,j} = V_{i,j}^{(0)}$ and $V_{i.j}^{(1)}$ be the same as above.
If $V_{i,j}^{(1)} = MV_{i,j}^{(1)}$, then $V_{i,j}$ admits a 
Blow-semialgebraic trivialisation consistent with 
the canonical filtration
$\{ V_{i,j}^{(0)} \supset V_{i,j}^{(1)} \}$ along $Q_{i,j}$.

In the case where $V_{i,j}^{(1)} \ne MV_{i,j}^{(1)}$,
let $V_{i,j}^{(2)}$ be the Zariski closure of 
$V_{i,j}^{(1)} \setminus MV_{i,j}^{(1)}$.
Then there is a finite partition
$Q_{i,j} = Q_{i,j,1} \cup \cdots \cup Q_{i,j,a(i,j)} \cup \breve{Q}_{i,j}$ 
which satisfies the following:

(I;1:3) Each $Q_{i,j,k}$ is a Nash open simplex of $\dim Q_{i,j}$.

(I;1:4) $\breve{Q}_{i,j}$ is a semialgebraic subset of $Q_{i,j}$
of dimension less than $\dim Q_{i,j}$.

(I;1:5) For each $k$, $V_{i,j}^{(2)} \cap q^{-1}(Q_{i,j,k})$ 
is Nash diffeomorphic to the direct product of some finite points 
in $V_{i,j}^{(2)} \cap q^{-1}(t)$, $t \in Q_{i,j,k}$, and $Q_{i,j,k}$.

Let $V_{i,j,k} = V_{i,j,k}^{(0)} = V_{i,j} \cap q^{-1}(Q_{i,j,k})$,
$V_{i,j,k}^{(1)} = V_{i,j}^{(1)} \cap q^{-1}(Q_{i,j,k})$,
$V_{i,j,k}^{(2)} = V_{i,j}^{(2)} \cap q^{-1}(Q_{i,j,k})$, and let
$W_{i,j,k}^{(2)}$ be the Zariski closure of
$V_{i,j,k}^{(1)} \setminus MV_{i,j,k}^{(1)}$ in $N \times Q_{i,j,k}$
for $1 \le k \le a(i,j)$.
By Lemma \ref{samedim} we have
$\dim V_{i,j,k}^{(2)} = \dim W_{i,j,k}^{(2)} < \dim V_{i.j.k}^{(1)}$
for $1 \le k \le a(i,j)$,
and we see that the filtration
$\{ V_{i,j,k}^{(0)} \supset V_{i,j,k}^{(1)} \supset V_{i,j,k}^{(2)} \}$
is a compatible one of $V_{i,j,k}$ for each $k$.
Therefore each $V_{i,j,k}$ admits a Blow-semialgebraic trivialisation
consistent with a compatible filtration
$\{ V_{i,j,k}^{(0)} \supset V_{i,j,k}^{(1)} \supset V_{i,j,k}^{(2)} \}$
along $Q_{i,j,k}$.

In the latter case, finiteness holds on Blow-semialgebraic triviality 
consistent with a compatible filtration 
for a family of zero-sets over $Q_{i,j}$ except the thin semialgebraic 
subset $\breve{Q}_{i,j}$ of $Q_{i,j}$.
We take a finite partition $\breve{Q}_{i,j} =
\breve{Q}_{i,j,1} \cup \cdots \cup \breve{Q}_{i,j,b(i,j)}$
such that each $\breve{Q}_{i,j,k}$ is a Nash open simplex,
and denote by $\Pi_i : \breve{\mathcal{M}}_{i,j,k} \to 
N \times \breve{Q}_{i,j,k}$ the restriction of $\Pi_i$ to
$\Pi_i^{-1}(N \times \breve{Q}_{i,j,k})$ for $1 \le k \le b(i,j)$.
Then each $MF_{\breve{Q}_{i,j,k}}^{-1}(0)$ admits 
a $\Pi_i$-Blow-semialgebraic trivialisation along $\breve{Q}_{i,j,k}$.
Therefore we can reduce our finiteness problem to the case
of the lower dimensional parameter space in this case.

\begin{rem}\label{remark61}
In $MV_{i,j,k}^{(1)}$, the Nash trivialisation of $V_{i,j,k}^{(2)}$
may not coincide with the semialgebraic trivialisation
induced by the Blow-Nash trivialisation related to $\beta_{i,j}$.
But this is not a problem for our definition of Blow-semialgebraic
triviality consistent with a compatible filtration.
\end{rem}

\noindent {\em Case} (I;2): There is a finite partition
$Q_i = Q_{i,1} \cup \cdots \cup Q_{i,w(i)}$ which satisfies the following:

(I;2:1) Each $Q_{i,j}$ is a Nash open simplex.

(I;2:2) For each $j$, there is a Nash simultaneous resolution 
$\beta_{i,j} : \mathcal{M}_{i,j} \to N \times Q_{i,j}$ of 
$V_i^{(1)} \cap q^{-1}(Q_{i,j})$ in $N \times Q_{i,j}$ such that
$(N \times Q_{i,j},V_i^{(1)} \cap q^{-1}(Q_{i,j}))$ admits 
a $\beta_{i,j}$-Blow-semialgebraic trivialisation along $Q_{i,j}$.

\begin{rem}\label{remark62}
The above $\beta_{i,j}$-Blow-semialgebraic triviality of
$(N \times Q_{i,j},V_i^{(1)} \cap q^{-1}(Q_{i,j}))$ over $Q_{i,j}$
is given as an extension of the semialgebraic triviality of a stratified
mapping $\beta_{i,j}|_{\mathcal{D}_{i,j}} : \mathcal{D}_{i,j} \to
\beta_{i,j}(\mathcal{D}_{i,j})$ over $Q_{i,j}$,
where $\mathcal{D}_{i,j}$ is the exceptional set of $\beta_{i,j}$.
The latter semialgebraic triviality follows from the semialgebraic
version of Thom's 2nd Isotopy Lemma with certain weak assumptions
(cf. Remark \ref{remark32}).
Therefore, for a semialgebraic subset
$A \subset \beta_{i,j}(\mathcal{D}_{i,j})$,
taking finite substratifications $\mathcal{S}(\mathcal{D}_{i,j})$
and $\mathcal{S}(\beta_{i,j}(\mathcal{D}_{i,j}))$ if necessary,
there is a finite partition
$Q_{i,j} = Q_{i,j,1} \cup \cdots \cup Q_{i,j,a(i,j)}$
which satisfies the following:

(1) Each $Q_{i,j,k}$ is a Nash open simplex.

(2) The stratified mappings
$\beta_{i,j}|_{\mathcal{D}_{i,j} \cap (q \circ \beta_{i,j})^{-1}(Q_{i,j,k})} 
: \mathcal{D}_{i,j} \cap (q \circ \beta_{i,j})^{-1}(Q_{i,j,k}) \to
\beta_{i,j}(\mathcal{D}_{i,j}) \cap q^{-1}(Q_{i,j,k})$ and
$q : \beta_{i,j}(\mathcal{D}_{i,j}) \cap q^{-1}(Q_{i,j,k}) \to Q_{i,j,k}$
satisfy the conditions of the 2nd Isotopy Lemma, and the stratification
$\mathcal{S}(\beta_{i,j}(\mathcal{D}_{i,j}) \cap q^{-1}(Q_{i,j,k}))$
is compatible with $A \cap q^{-1}(Q_{i,j,k})$ for $1 \le k \le a(i,j).$

Therefore, each $(N \times Q_{i,j,k},V_i^{(1)} \cap q^{-1}(Q_{i,j,k}))$
admits a $\beta_{i,j}$-Blow-semialgebraic trivialisation along $Q_{i,j,k}$,
where this $\beta_{i,j}$ is the restriction of the original
$\beta_{i,j}$ over $Q_{i,j,k}$.
The induced semialgebraic triviality of $N \times Q_{i,j,k}$
trivialises also $A \cap q^{-1}(Q_{i,j,k})$.
\end{rem}

We keep the same notations $V_{i,j} = V_{i,j}^{(0)}$,
$V_{i,j}^{(1)}$, $V_{i,j}^{(2)}$, $V_{i,j,k} = V_{i,j,k}^{(0)}$,
$V_{i,j,k}^{(1)}$, $V_{i,j,k}^{(2)}$ and $W_{i,j,k}^{(2)}$ as above.

If $V_{i,j}^{(1)} = MV_{i,j}^{(1)}$, then $V_{i,j}$ admits
a Blow-semialgebraic trivialisation consistent 
with the canonical filtration $\{ V_{i,j}^{(0)} \supset V_{i,j}^{(1)} \}$
along $Q_{i,j}$.

If $V_{i,j}^{(1)} \setminus MV_{i,j}^{(1)} \ne \emptyset$,
then it is 0-dimensional or 1-dimensional over $Q_{i,j}$.
In the 0-dimensional case, there is a finite partition
$Q_{i,j} = Q_{i,j,1} \cup \cdots \cup Q_{i,j,a(i,j)} \cup \breve{Q}_{i,j}$ 
which satisfies the following:

(I;2:3) Each $Q_{i,j,k}$ is a Nash open simplex of $\dim Q_{i,j}$.

(I;2:4) $\breve{Q}_{i,j}$ is a semialgebraic subset of $Q_{i,j}$
of dimension less than $\dim Q_{i,j}$.

(I;2:5) Each $V_{i,j,k}^{(2)}$ is Nash diffeomorphic to the direct product 
of some finite points in $V_{i,j}^{(2)} \cap q^{-1}(t)$, 
$t \in Q_{i,j,k}$, and $Q_{i,j,k}$.

\noindent Similarly to Case (I;1), we see that each $V_{i,j,k}$ 
admits a Blow-semialgebraic trivialisation consistent
with a compatible filtration along $Q_{i,j,k}$.

In the case where $V_{i,j}^{(1)} \setminus MV_{i,j}^{(1)}$
is 1-dimensional over $Q_{i,j}$ and $L_{i,j}^{(1)} =
MV_{i,j}^{(1)} \cap \overline{V_{i,j}^{(1)} \setminus MV_{i,j}^{(1)}}$
is non-empty, $L_{i,j}^{(1)}$ is a semialgebraic subset of $MV_{i,j}^{(1)}$
which is 0-dimensional and semialgebraically trivial over $Q_{i,j}$.
Thanks to the desingularisation construction by Hironaka,
$L_{i,j}^{(1)}$ must be contained in $\beta_{i,j}(\mathcal{D}_{i,j})$.
Taking this fact and Remark \ref{remark62} into consideration,
we can take a finite partition 
$Q_{i,j} = Q_{i,j,1} \cup \cdots \cup Q_{i,j,a(i,j)} \cup \breve{Q}_{i,j}$ 
which satisfies the following:

(I;2:6) Each $Q_{i,j,k}$ is a Nash open simplex of $\dim Q_{i,j}$.

(I;2:7) $\breve{Q}_{i,j}$ is a semialgebraic subset of $Q_{i,j}$
of dimension less than $\dim Q_{i,j}$.

(I;2:8) For $1 \le k \le a(i,j)$, $(N \times Q_{i,j,k},V_{i,j,k}^{(1)})$
admits a $\beta_{i,j}$-Blow-semialgebraic trivialisation along $Q_{i,j,k}$.

(I;2:9) For $1 \le k \le a(i,j)$, there is a Nash simultaneous resolution
$\gamma_{i,j,k} : \mathcal{M}_{i,j,k} \to N \times Q_{i,j,k}$ of
$V_{i,j,k}^{(2)}$ in $N \times Q_{i,j,k}$ such that
$(N \times Q_{i,j,k},V_{i,j,k}^{(2)})$ admits
a $\gamma_{i,j,k}$-Blow-Nash trivialisation along $Q_{i,j,k}$.

(I;2:10) If $L_{i,j}^{(1)} \ne \emptyset$, then over 
$L_{i,j}^{(1)} \cap q^{-1}(Q_{i,j,k})$,
the semialgebraic trivialisation of $V_{i,j,k}^{(1)}$ induced by
the $\beta_{i,j}$-Blow-semialgebraic trivialisation in (I;2:8) 
coincides with the semialgebraic one of $V_{i,j,k}^{(2)}$ 
induced by the $\gamma_{i,j,k}$-Blow-Nash trivialisation in (I;2:9).

By construction and Lemma \ref{samedim} we have
$\dim V_{i,j,k}^{(0)} > \dim V_{i,j,k}^{(1)} > \dim V_{i,j,k}^{(2)}$
for $1 \le k \le a(i,j)$.
Therefore, in the case where $V_{i,j} \setminus MV_{i,j}$ is 1-dimensional
over $Q_{i,j}$ and $V_{i,j,k}^{(2)} = MV_{i,j,k}^{(2)}$,
we see that $V_{i,j,k}$ admits a Blow-semialgebraic trivialisation
consistent with a compatible filtration
$\{ V_{i,j,k}^{(0)} \supset V_{i,j,k}^{(1)} \supset V_{i,j,k}^{(2)} \}$
along $Q_{i,j,k}$.

It remains to consider the case where $V_{i,j} \setminus MV_{i,j}$
is 1-dimensional over $Q_{i,j}$ and $V_{i,j,k}^{(2)} \ne MV_{i,j,k}^{(2)}$.
Let $V_{i,j,k}^{(3)}$ denote the Zariski closure of
$V_{i,j,k}^{(2)} \setminus MV_{i,j,k}^{(2)}$ in $N \times Q_{i,j,k}$.
Then there is a finite partition $Q_{i,j,k} =
Q_{i,j,k,1} \cup \cdots \cup Q_{i,j,k,c(i,j,k)} \cup \breve{Q}_{i,j,k}$ 
which satisfies the following:

(I;2:11) Each $Q_{i,j,k,s}$ is a Nash open simplex of $\dim Q_{i,j,k}$.

(I;2:12) $\breve{Q}_{i,j,k}$ is a semialgebraic subset of $Q_{i,j,k}$
of dimension less than $\dim Q_{i,j,k}$.

(I;2:13) Each $V_{i,j,k,s}^{(3)}$ is Nash diffeomorphic to the direct product 
of some finite points in $V_{i,j,k}^{(3)} \cap q^{-1}(t)$, 
$t \in Q_{i,j,k,s}$, and $Q_{i,j,k,s}$.

Let

\vspace{3mm}

\centerline{$V_{i,j,k,s} = V_{i,j,k,s}^{(0)} 
= V_{i,j,k} \cap q^{-1}(Q_{i,j,k,s})$,
$V_{i,j,k,s}^{(1)} = V_{i,j,k}^{(1)} \cap q^{-1}(Q_{i,j,k,s})$,}

\vspace{3mm}

\centerline{$V_{i,j,k,s}^{(2)} = V_{i,j,k}^{(2)} \cap q^{-1}(Q_{i,j,k,s})$,
$V_{i,j,k,s}^{(3)} = V_{i,j,k}^{(3)} \cap q^{-1}(Q_{i,j,k,s})$.}

\vspace{3mm}

\noindent Then it follows from the construction that
$\dim V_{i,j,k,s}^{(0)} > \dim V_{i,j,k,s}^{(1)} >
\dim V_{i,j,k,s}^{(2)} > \dim V_{i,j,k,s}^{(3)}$
for $1 \le s \le c(i,j,k)$.
Therefore each $V_{i,j,k,s}$ admits a Blow-semialgebraic
trivialisation consistent with a compatible filtration
$\{ V_{i,j,k,s}^{(0)} \supset V_{i,j,k,s}^{(1)} \supset
V_{i,j,k,s}^{(2)} \supset V_{i,j,k,s}^{(3)} \}$ along $Q_{i,j,k,s}$.

\vspace{2mm}

We next consider the case where $V_i^{(1)} \cap MV_i \ne \emptyset$.
Then there are 3 possibilities for the dimension of 
$V_i^{(1)} \setminus MV_i$.
Therefore we subdivide the second case into the following 3 cases:

\vspace{2mm}

(II;0) \ $\dim (V_i^{(1)} \setminus MV_i) = \dim Q_i$.

(II;1) \ $\dim (V_i^{(1)} \setminus MV_i) = \dim Q_i + 1$.

(II;2) \ $\dim (V_i^{(1)} \setminus MV_i) = \dim Q_i + 2$.

\vspace{2mm}

Let us keep the same notations $V_{i,j} = V_{i,j}^{(0)}$, $V_{i,j}^{(1)}$,
$V_{i,j}^{(2)}$, $L_{i,j}^{(1)}$, $V_{i,j,k} = V_{i,j,k}^{(0)}$, 
$V_{i,j,k}^{(1)}$, $V_{i,j,k}^{(2)}$, $V_{i,j,k}^{(3)}$, 
$V_{i,j,k,s} = V_{i,j,k,s}^{(0)}$, $V_{i,j,k,s}^{(1)}$, 
$V_{i,j,k,s}^{(2)}$ and $V_{i,j,k,s}^{(3)}$
as above.

\vspace{2mm}

\noindent {\em Case} (II;0): There is a finite partition
$Q_i = Q_{i,1} \cup \cdots \cup Q_{i,w(i)} \cup \breve{Q}_i$ 
which satisfies the following:

(II;0:1) Each $Q_{i,j}$ is a Nash open simplex of $\dim Q_i$.

(II;0:2) $\breve{Q}_i$ is a semialgebraic subset of $Q_i$
of dimension less than $\dim Q_i$.

(II;0:3) For each $j$, $V_i^{(1)} \cap q^{-1}(Q_{i,j})$ 
is Nash diffeomorphic to the direct product of some finite points 
in $V_i^{(1)} \cap q^{-1}(t)$, $t \in Q_{i,j}$, and $Q_{i,j}$.

Therefore we see that each $V_{i,j}$ admits a Blow-semialgebraic 
trivialisation consistent with a compatible filtration
$\{ V_{i,j}^{(0)} \supset V_{i,j}^{(1)} \}$ along $Q_{i,j}$.

\vspace{2mm}

\noindent {\em Case} (II;1): Let $L_i^{(0)} =
MV_i \cap \overline{V_i \setminus MV_i}$.
Then it is non-empty, and is a semialgebraic subset of $MV_i$
which is 0-dimensional and semialgebraically trivial over $Q_i$.
Therefore it follows from Remark \ref{keyremark} that
there is a finite partition 
$Q_i = Q_{i,1} \cup \cdots \cup Q_{i,w(i)} \cup \breve{Q}_i$ 
which satisfies the following:

(II;1:1) Each $Q_{i,j}$ is a Nash open simplex of $\dim Q_i$.

(II;1:2) $\breve{Q}_i$ is a semialgebraic subset of $Q_i$
of dimension less than $\dim Q_i$.

(II;1:3) For each $j$, there is a Nash simultaneous resolution 
$\beta_{i,j} : \mathcal{M}_{i,j} \to N \times Q_{i,j}$ of 
$V_{i,j}^{(1)}$ in $N \times Q_{i,j}$ such that
$(N \times Q_{i,j},V_{i,j}^{(1)})$ admits 
a $\beta_{i,j}$-Blow-Nash trivialisation along $Q_{i,j}$.

(II;1:4) Over $L_i^{(0)} \cap q^{-1}(Q_{i,j})$,
the semialgebraic trivialisation of $V_i^{(0)}$ induced by
the $\Pi_i$-Blow-semialgebraic trivialisation of $MV_i^{(0)}$ 
coincides with the semialgebraic one of $V_{i,j}^{(1)}$ 
induced by the $\beta_{i,j}$-Blow-Nash trivialisation in (II;1:3).

In the case where $V_{i,j}^{(1)} = MV_{i,j}^{(1)}$, $V_{i,j}$
admits a Blow-semialgebraic trivialisation consistent with 
a compatible filtration
$\{ V_{i,j}^{(0)} \supset V_{i,j}^{(1)} \}$ along $Q_{i,j}$.

In the case where $V_{i,j}^{(1)} \ne MV_{i,j}^{(1)}$,
there is a finite partition
$Q_{i,j} = Q_{i,j,1} \cup \cdots \cup Q_{i,j,a(i,j)} \cup \breve{Q}_{i,j}$ 
which satisfies the following:

(II;1:5) Each $Q_{i,j,k}$ is a Nash open simplex of $\dim Q_{i,j}$.

(II;1:6) $\breve{Q}_{i,j}$ is a semialgebraic subset of $Q_{i,j}$
of dimension less than $\dim Q_{i,j}$.

(II;1:7) For each $k$, $V_{i,j}^{(2)} \cap q^{-1}(Q_{i,j,k})$ 
is Nash diffeomorphic to the direct product of some finite points 
in $V_{i,j}^{(2)} \cap q^{-1}(t)$, $t \in Q_{i,j,k}$, and $Q_{i,j,k}$.

Similarly to case (I;1), we can see that each $V_{i,j,k}$ admits 
a Blow-semialgebraic trivialisation consistent with a compatible filtration
$\{ V_{i,j,k}^{(0)} \supset V_{i,j,k}^{(1)} \supset V_{i,j,k}^{(2)} \}$
along $Q_{i,j,k}$.

\vspace{2mm}

\noindent {\em Case} (II;2): Let $L_i^{(0)}$ be the same as above.
Then it is a semialgebraic subset of $MV_i$
which is 0- or 1-dimensional and semialgebraically trivial over $Q_i$.
Similarly to case (II;1), there is a finite partition 
$Q_i = Q_{i,1} \cup \cdots \cup Q_{i,w(i)} \cup \breve{Q}_i$ 
which satisfies the same conditions as (II;1:1) - (II;1:4), 
replacing the $\beta_{i,j}$-Blow-Nash trivialisation 
in (II;1:3) with a $\beta_{i,j}$-Blow-semialgebraic trivialisation.

If $V_{i,j}^{(1)} = MV_{i,j}^{(1)}$, then $V_{i,j}$
admits a Blow-semialgebraic trivialisation consistent with 
a compatible filtration
$\{ V_{i,j}^{(0)} \supset V_{i,j}^{(1)} \}$ along $Q_{i,j}$.

If $V_{i,j}^{(1)} \setminus MV_{i,j}^{(1)} \ne \emptyset$,
then it is 0-dimensional or 1-dimensional over $Q_{i,j}$.
In the 0-dimensional case, there is a finite partition
$Q_{i,j} = Q_{i,j,1} \cup \cdots \cup Q_{i,j,a(i,j)} \cup \breve{Q}_{i,j}$ 
which satisfies the same conditions as (I;2:3) - (I;2:5).
Therefore each $V_{i,j,k}$ admits
a Blow-semialgebraic trivialisation consistent with a compatible filtration 
$\{ V_{i,j,k}^{(0)} \supset V_{i,j,k}^{(1)} \supset V_{i,j,k}^{(2)} \}$
along $Q_{i,j,k}$.

We last consider the case where $V_{i,j}^{(1)} \setminus MV_{i,j}^{(1)}$
is 1-dimensional over $Q_{i,j}$.
If $L_{i,j}^{(1)} \ne \emptyset$,
it is a semialgebraic subset of $MV_{i,j}^{(1)}$
which is 0-dimensional and semialgebraically trivial over $Q_{i,j}$.
Using a similar argument to Cases (I;2) based on Remarks \ref{remark62},
we can take a finite partition 
$Q_{i,j} = Q_{i,j,1} \cup \cdots \cup Q_{i,j,a(i,j)} \cup \breve{Q}_{i,j}$ 
which satisfies the same conditions as (I;2:6) - (I;2:10).
In addition, we can assume also the following:

(II;2:1) Over $L_i^{(0)} \cap q^{-1}(Q_{i,j,k})$,
the semialgebraic trivialisation of $V_{i,j.k}^{(0)}$ induced by
the $\Pi_i$-Blow-semialgebraic trivialisation of $MV_{i,j,k}^{(0)}$ 
coincides with the semialgebraic one of $V_{i,j,k}^{(1)}$ 
induced by the $\beta_{i,j}$-Blow-semialgebraic trivialisation
given in the condition corresponding to (II;1:3).

Similarly to Case (I;2), in the case where 
$V_{i,j} \setminus MV_{i,j}$ is 1-dimensional over $Q_{i,j}$ 
and $V_{i,j,k}^{(2)} = MV_{i,j,k}^{(2)}$,
we see that $V_{i,j,k}$ admits a Blow-semialgebraic trivialisation
consistent with a compatible filtration
$\{ V_{i,j,k}^{(0)} \supset V_{i,j,k}^{(1)} \supset V_{i,j,k}^{(2)} \}$
along $Q_{i,j,k}$.

In the case where $V_{i,j} \setminus MV_{i,j}$
is 1-dimensional over $Q_{i,j}$ and $V_{i,j,k}^{(2)} \ne MV_{i,j,k}^{(2)}$,
there is a finite partition $Q_{i,j,k} =
Q_{i,j,k,1} \cup \cdots \cup Q_{i,j,k,c(i,j,k)} \cup \breve{Q}_{i,j,k}$ 
which satisfies the same conditions as (I;2:11) - (I;2:13).
Therefore, similarly to Case (I;2), we can see 
that each $V_{i,j,k,s}$ admits a Blow-semialgebraic
trivialisation consistent with a compatible filtration
$\{ V_{i,j,k,s}^{(0)} \supset V_{i,j,k,s}^{(1)} \supset
V_{i,j,k,s}^{(2)} \supset V_{i,j,k,s}^{(3)} \}$ along $Q_{i,j,k,s}$.

\vspace{2mm}

In any case, removing a lower dimensional semialgebraic subset
from the original parameter space $Q_i$, $r + 1 \le i \le v$, 
if necessary, we proved that finiteness property holds on 
a Blow-semialgebraic triviality consistent with a compatible filtration 
for a family of 3-dimensional algebraic sets.
Let $T_1 = Q_{r+1} \cup \cdots \cup Q_v$.
The union of all the removed semialgebraic subsets is
a lower dimensional semialgebraic subset of $J$.
We take the Zariski closure $J_1$ of the union.
Taking a finite subdivision if necessary, the same finiteness property 
as above still holds outside $J_1$ in any case.
We apply Proposition \ref{mainpart} to the family 
$\{ f_t^{-1}(0) \}_{t \in J_1}$ similarly to the beginning of this proof.
Then we ignore the finiteness property over $J_1 \setminus T_1$.
Because we need not treat again the finiteness over $J \setminus T_1$.
In this way we can deduce the problem to the lower dimensional case.
Therefore we can finish the proof of the Main Theorem by induction 
on the dimension of the parameter space.
\end{proof}


\bigskip
\section{Finiteness for 3-dimensional Nash sets.}
\label{nashsets}
\medskip

In this section we give a finiteness theorem
on Blow-semialgebraic triviality consistent with a Nash compatible
filtration for a family of 3-dimensional Nash sets.
In order to show the finiteness theorem for a family of the main parts
of 3-dimensional algebraic sets, we considered the comlexification in \S 5.
Therefore the finiteness theorem corresponding to Process VII holds
only for a family of 3-dimensional algebraic sets defined over
a nonsingular algebraic variety.
As a result, we have to modify Process VIII.
We describe the proof using a different method.

We first recall the Artin-Mazur Theorem (M. Artin and B. Mazur 
\cite{artinmazur}, M. Coste, J.M. Ruiz and M. Shiota \cite{costeruizshiota}, 
M. Shiota \cite{shiota1}) for a family of Nash mappings.

\begin{thm}\label{artinmazur} (Artin-Mazur Theorem).
Let $P \subset \R^p$ and $T \subset \R^q$ be nonsingular algebraic 
varieties, and let $g : P \times T \to \R^k$ be a Nash mapping.
Then there exits a Nash mapping $h : P \times T \to \R^b$
with the following property:

\vspace{1mm}

Let $\tau : P \times T \times \R^k \times \R^b \to P \times T$
and $\pi : P \times T \times \R^k \times \R^b \to \R^k$ 
be the canonical projections, 
let $G = (g,h) : P \times T \to \R^k \times \R^b$ be the Nash mapping
defined by $G(x,t) = (g(x,t),h(x,t))$,
and let $X$ be the Zariski closure of graph $G$.
Then there is a union $L$ of some connected components of $X$ with
$\dim L = \dim X$ such that $\tau |_L : L \to P \times T$
is a $t$-level preserving Nash diffeomorphism
and $(\pi |_L) \circ (\tau |_L)^{-1} = g$.
\end{thm}

\begin{rem}\label{remark71}
In the Artin-Mazur theorem above, $L$ is contained 
in the smooth part of the algebraic set $X$, denoted by $Reg(X)$.
\end{rem}

Let $N$ be a Nash manifold in $\R^m$, and let $J$ be 
a semialgebraic set in $\R^a$.
Let $f_t : N \to \R^k$ ($t \in J$) be a Nash mapping
such that $\dim f_t^{-1}(0) \le 3$ for $t \in J$.
Assume that $F$ is a Nash mapping.
Then we have

\vspace{3mm}

\noindent {\bf Theorem V.} {\em There exists a finite partition
$J = Q_1 \cup \cdots \cup Q_u$
which satisfies the following conditions}:

(1) {\em Each $Q_i$ is a Nash open simplex.}

(2) {\em For each $i$ where $\dim f_t^{-1}(0) \le 3$ and
$\dim f_t^{-1}(0) \cap S(f_t) \ge 1$ over $Q_i$,
there are a nonsingular algebraic variety $\hat{N}_i$
and a $t$-level preserving Nash embedding
$\alpha_i : N \times Q_i \to \hat{N}_i \times Q_i$  
such that $\alpha_i(F_{Q_i}^{-1}(0))$ admits a Blow-semialgebraic 
trivialisation consistent with a Nash compatible filtration along $Q_i$.

In the case where $\dim f_t^{-1}(0) \le 2$ over $Q_i$
or $\dim f_t^{-1}(0) \cap S(f_t) \le 0$
over $Q_i$,
$(N \times Q_i,F_{Q_i}^{-1}(0))$ admits
a trivialisation listed in table (*).}

\begin{proof}
Note that Lemma \ref{dimension2} holds also for a family of Nash mappings $F$.
By Process I in \S 3, it suffices to consider the case where 
$\dim f_t^{-1}(0) = 3$ and $\dim f_t^{-1}(0) \cap S(f_t) \ge 1$ 
for any $t \in J$.
Under this assumption, we have the following lemma:

\begin{lem}\label{lemma71}
There exists a finite partition
$J = Q_1 \cup \cdots \cup Q_u$
which satisfies the following conditions:

(1) Each $Q_i$ is a Nash open simplex.

(2) For each $i$, there are a nonsingular algebraic variety $\hat{N}_i$,
a $t$-level preserving Nash embedding $\alpha_i : N \times Q_i \to
\hat{N}_i \times Q_i$ and a Nash simultaneous resolution 
$\hat{\Pi}_i : \hat{\mathcal{M}}_i \to \hat{N}_i \times Q_i$ 
of $\alpha_i(F_{Q_i}^{-1}(0))$ in $\hat{N}_i \times Q_i$ such that
$\alpha_i(MF_{Q_i}^{-1}(0))$ admits a $\hat{\Pi}_i$-Blow-semialgebraic 
trivialisation along $Q_i$.
\end{lem}

\begin{proof}
By Theorem \ref{lojasiewicz}, there is a partition of $J$ into
finite Nash manifolds $Q_i$, $i = 1, \cdots , e$, such that each $Q_i$
is Nash diffeomorphic to an open simplex in some Euclidean space.
As stated in Process VIII, every Nash manifold is Nash diffeomorphic 
to a nonsingular (affine) algebraic variety.
Therefore we may assume that $N$ and $Q_i$, $1 \le i \le e$, 
are nonsingular algebraic varieties.
Let $N \subset \R^m$.

We apply the Artin-Mazur Theorem to each Nash mapping
$F_{Q_i} : N \times Q_i \to \R^k$, $i = 1, \cdots , e$.
Then there are an algebraic variety 
$X_i \subset N \times \R^k \times \R^{b_i} \times Q_i$,
a union $L_i$ of some connected components of $X_i$
with $\dim L_i = \dim X_i$ and $L_i \subset Reg(X_i)$,
and a $t$-level preserving Nash diffeomorphism 
$\tau_i : L_i \to N \times Q_i$ such that 
$(\pi |_{L_i}) \circ \tau_i^{-1} = F_{Q_i}$, where
$\pi : N \times \R^k \times \R^{b_i} \times Q_i \to \R^k$
is the canonical projection.
Let $\hat{N}_i = N \times \R^k \times \R^{b_i}$, and
let $Z_i$ be the Zariski closure of $W_i = L_i \cap \pi^{-1}(0)$ 
in $\hat{N}_i \times Q_i$.
Then $Z_i$ is an algebraic subset of $X_i \cap \pi^{-1}(0)$.
Note that $W_i$ is the union of connected components
of $Z_i$ contained in $L_i$.
Next let $\beta_i : \hat{N_i} \times Q_i \to Q_i$ be the
canonical projection.
By the proof of Proposition \ref{mainpart}, we see that there is 
a finite partition of $Q_i = Q_{i,1} \cup \cdots \cup Q_{i,u(i)}$ 
which satisfies the following:

(1) Each $Q_{i,j}$ is a Nash open simplex.

(2) For each $j$, there is a Nash simultaneous resolution
$\Pi_{i,j} : \mathcal{M}_{i,j} \to \hat{N_i} \times Q_{i,j}$
of $Z_{i,j} = Z_i \cap \beta_i^{-1}(Q_{i,j})$
in $\hat{N_i} \times Q_{i,j}$ such that 
$MZ_{i,j}$ admits a $\Pi_{i,j}$-Blow semialgebraic trivialisation
along $Q_{i,j}$.

For $1 \le j \le u(i)$, let $W_{i,j} = W_i \cap \beta_i^{-1}(Q_{i.j})$.
Since $W_i = Z_i \cap L_i$ and
$MW_{i,j} = MZ_{i,j} \cap W_{i,j}$ for $1 \le j \le u(i)$,
each $\Pi_{i,j}$-Blow-semialgebraic trivialisation of $MZ_{i,j}$
induces a Blow-semialgebraic trivialisation 
of $MW_{i,j}$ along $Q_{i,j}$.
Note that $MW_{i,j} = \tau_i^{-1}(MF_{Q_{i,j}}^{-1}(0))$.
Therefore the statement of the lemma follows.
\end{proof}

Thanks to Hironaka \cite{hironaka1, hironaka3} and Bierstone-Milman
\cite{bierstonemilman1, bierstonemilman2, bierstonemilman3},
the desingularisation theorem holds also in the Nash category.
In addition, a Nash compatible filtration of a  Nash set is
preserved by a Nash diffeomorphism.
Therefore Theorem IV follows from a similar argument 
to the proof of Main Theorem with Lemma \ref{lemma71}.
\end{proof}


\bigskip
\section{Finiteness on semialgebraic types of polynomial maps over $\R^2$.}
\label{plynomialmaps}
\medskip

Let $\K = \R$ or $\C$.
We denote by $P^d_{\K}(n,p)$ the set of polynomial mappings
from $\K^n$ to $\K^p$ of degree $\le d$.
We say that two polynomial functions $f$, $g : \K^n \to \K$
are {\em topologically equivalent}, if there is a homeomorphism
$\sigma : \K^n \to \K^n$ such that $f = g \circ \sigma$.
In the real case, if we can take the $\sigma$ as a semialgebraic
homeomorphism, we say that $f$ and $g$ are {\em semialgebraically
equivalent}.
We say that two polynomial mappings $f$, $g : \K^n \to \K^p$,
$p \ge 2$, are {\em topologically equivalent}, if there are homeomorphisms
$\sigma : \K^n \to \K^n$ and $\tau : \K^p \to \K^p$
such that $\tau \circ f = g \circ \sigma$.
Similarly to the function case, we say that real polynomial
mappings $f$ and $g$ are {\em semialgebraically equivalent} 
in the case where $\sigma$ and $\tau$ are semialgebraic homeomorphisms.

We review the results on finiteness of topological or semialgebraic
types of polynomial functions or mappings.
Concerning the topological types of polynomial functions,
a finiteness theorem is shown by T. Fukuda.
He proves in \cite{fukuda1} that the number of topological types
appearing in $P^d_{\K}(n,1)$ for $\K = \R$ or $\C$ is finite.
The real result is strengthened by R. Benedetti and M. Shiota
\cite{benedettishiota}.
They give a finiteness theorem on semialgebraic equivalence.

On the other hand, some finiteness theorems on topological equivalence 
are also known for polynomial mappings of two variables.
K. Aoki \cite{aoki}, for instance, proves a finiteness theorem
for plane-to-plane mappings in $P^d_{\K}(2,2)$ for $\K = \R$ or $\C$.
Aoki's result in the complex case is generalised by C. Sabbah
\cite{sabbah1} to a finiteness theorem for polynomial mappings
in $P^d_{\C}(2,p)$.
But finiteness does not hold in general for polynomial mappings
of more than $2$ variables.
As mentioned in \S 4, I. Nakai \cite{nakai} constructs
a polynomial family of polynomial mappings of at least $3$ variables
in which topological moduli appear.

In this section we make a remark on finiteness
for real polynomial mappings of two variables.
Using the arguments discussed in \S 5, we can show the following
result more easily than Proposition \ref{mainpart}.

\vspace{3mm}

\noindent {\bf Theorem VI.} {\em The number of semialgebraic types
appearing in $P^d_{\R}(2,p)$ is finite.}

\end{document}